\documentclass{IOS-Book-Article}

\usepackage{mathptmx}

\usepackage{mathtools}
\usepackage{amssymb, hyperref}
 
\usepackage{ amsmath,amssymb,amsbsy,amsfonts, latexsym,amsopn,amstext, amsxtra,euscript,amscd}
\usepackage{amsthm}

\usepackage{cite}

\usepackage[alphabetic,   msc-links]{amsrefs}

\hypersetup{
  colorlinks   = true, %Colours links instead of ugly boxes
  urlcolor     = blue, %Colour for external hyperlinks
  linkcolor    = blue, %Colour of internal links
  citecolor   = red %Colour of citations
}

 \def\cR{\mathcal R}

\newtheorem{thm}{Theorem}

\newtheorem {prop}{Proposition}
\newtheorem {lem}{Lemma}
\newtheorem {defi}{Definition}
\newtheorem {exa}{Example}

\newtheorem{alg}{Algorithm}
\newtheorem {rem}{Remark}
\newtheorem {cor}{Corollary}

\DeclareMathOperator\Aut{Aut}

\newcommand\Z{\mathbb Z}
\newcommand\Q{\mathbb Q}
\newcommand\R{\mathbb R}
\newcommand\C{\mathbb C}

\newcommand\bQ{\overline{\mathbb Q}}

\def\P{\mathbb P}

\newcommand\M{{\mathcal M}}

\newcommand\X{\mathcal X}            %\newcommand\X{\mathfrak X}
\newcommand\B{\mathcal B}

\newcommand\w{\mathfrak w}

\newcommand\iso{\cong}

\newcommand\D{\Delta}

\newcommand\<{\langle}

\def\a{\alpha}

\def\deg{\mbox{deg }}

\def\mH{\mathcal H}

\def\iso{\equiv}

\DeclareFontFamily{U}{wncy}{}
    \DeclareFontShape{U}{wncy}{m}{n}{<->wncyr10}{}
    \DeclareSymbolFont{mcy}{U}{wncy}{m}{n}

\DeclareMathSymbol{\Sh}{\mathord}{mcy}{"58}

\def\p{\mathfrak p}
\def\O{\mathcal O}

\begin{document}
\begin{frontmatter}              % The preamble begins here.

%\pretitle{Pretitle}
\title{Heights on algebraic curves}
%\runningtitle{Heights on algebraic curves}
%\subtitle{Subtitle}

\author[A]{\fnms{T. Shaska}\snm{}%
%\thanks{Corresponding Author: Oakland University, Rochester, MI, USA;  E-mail: shaska@oakland.edu}
}
\author[B]{\fnms{L. Beshaj} \snm{}
\thanks{Part of this paper was written when this author was visiting Department of Mathematics at Princeton University. This authors wants to thanks Princeton University for their hospitality.}
}

\runningauthor{Beshaj/Shaska}
\address[A]{Oakland University, Rochester, MI, USA;  E-mail: shaska@oakland.edu}
\address[B]{Oakland University, Rochester, MI, USA;  E-mail: beshaj@oakland.edu}

\begin{abstract}
In these lectures we cover basics of the theory of heights starting with the heights in the projective space, heights of polynomials, and heights of the algebraic curves. We define the minimal height of binary forms and moduli height for algebraic curves and prove that   the moduli height of superelliptic curves $\mH (f) \leq c_0 \tilde H (f)$ where $c_0$ is a constant and $\tilde H$  the minimal height of the corresponding binary form.  For genus $g=2$ and 3 such constant is explicitly determined.  
 Furthermore, complete lists of curves of genus 2 and genus 3 hyperelliptic curves with height 1 are computed. 
\end{abstract}


%\subjclass[2000]{11G30 \and 11G50 \and 14G40}

\end{frontmatter}

%\thispagestyle{empty}
%\pagestyle{empty}

%\setcounter{tocdepth}{1}
%\tableofcontents

\section*{Introduction}
%*********************************************************************

The heights of points on Abelian varieties have been used to prove important results on the theory of rational points on algebraic curves. In these lectures we give a quick review of some of the basic results of the classical theory of heights  and  introduce some new concepts of heights for algebraic curves with the intention of determining equations of curves with minimal height. 

The material of the first lecture is classic and can be found in any of the excellent books \cite{bombieri, silv-book, lang-2, se-1}.   We define affine and projective  heights on the projective space,   multiplicative and logarithmic heights, and absolute heights.  We describe Northcott's theorem, Kronecker's theorem, and Segre embedding. Our main goal is to investigate how the height of a point changes under a change of coordinate. We describe the formula for changing coordinates in Thm.~\ref{ch-coord}.

In the second lecture we cover the heights of polynomials, Gauss lemma on heights, Gelfand's inequality, and bounds on heights of homogenous polynomials acting on them by linear transformations on variables.  The main focus of this lecture is on the heights of binary forms. These are interesting polynomials because they give equations of hyperelliptic and superelliptic curves which are the focus of this Summer School. For any binary form $f$ we   provide bounds for $f^M$ when $M\in SL_2 (K)$.  This leads to the definition of  the \textit{minimal height} $\tilde H (f)$ and \textit{moduli height} $\mH (f)$ for binary forms. We prove that $\mH (f) \leq c_0 \tilde H (f)^{n_0}$ for any binary form $f$, where $c_0$ and $n_0$ are  constants depending only on the degree of the binary form $f$.

In the third lecture we focus on heights of algebraic curves.  
Our main focus is in providing equations for the algebraic curves with "small" coefficients as continuation of our previous work \cite{beshaj-1, sh-1 }.  Hence, the concept of height is the natural concept to be used.   For a genus $g\geq 2$ algebraic curve $\X_g$ defined over an algebraic number field $K$ we define the height $H_K (\X_g)$ and show that this is well-defined.  This is basically the minimum height among all curves which are isomorphic to $\X_g$ over $K$.  $\bar H_K (\X_g)$ is the height over the algebraic closure $\bar K$.  It must be noticed that our definition is on the isomorphism class of the curve and not on some equation of the curve. 
 We provide an algorithm to determine the height of a curve $C$ provided some equation for $C$.  This algorithm   is rather inefficient, but can be used for $g=2$ and $g=3$ hyperelliptic curves when the   coefficients of the initial equation of $C$ are not too large. 

The moduli height of a curve is the height in the projective space of the moduli point corresponding to the curve.  We prove that for a given constant $c$ there are only finitely many curves (up to isomorphism) of moduli height $\leq c$. 

A natural applications of the results of lecture 3 would be superelliptic curves.  Such curves have equation $y^n = f(x)$ and their isomorphism classes are determined by the $GL_2 (K)$-orbits on the space of degree $d$ binary forms $V_d$, where $d = \deg f$. Hence, we apply the results from lecture 2 to study the heights of such curves.

For a genus 2 curve with equation $y^2=f(x)$ the moduli height is bounded as follows
\[ \mH(f) \leq 2^{28} \cdot 3^9 \cdot 5^5 \cdot 7 \cdot 11 \cdot 13 \cdot  17 \cdot 43    \cdot H(f)^{10} \]
Moreover we show that there are precisely 230 genus 2 curves with height 1, from which 186 have automorphism group of order 2, 28 of them have automorphism group isomorphic to the Klein 4-group, and the rest have automorphism group of order $ > 4$.  All these curves are listed in Tables 1-4.

Similar computations are done for genus 3 hyperelliptic curves where $GL_2 (K)$-invariants $t_1, \dots , t_6$ are used as defined in \cite{sh-4}.  In the last part of this lecture we present some open problems and conjectures. 

For more references and classical results on the theory of heights the reader should check these timeless books   \cite{lang-2, se-1, se-2, se-3, Fa, bombieri, silv-book }.      \\

\noindent \textbf{Notation: } Throughout this paper by a curve we mean a smooth, irreducible algebraic curve.  Unless otherwise noted a "curve" $C$ means the isomorphism class of $C$ over some  field $K$.   We fix the following notation throughout this paper.\\

$K$ a number field,

$\O_K$ the ring of integers of $K$,

$v$ an absolute value of $K$,

$M_K$ the set of all absolute values  of $K$,

$M_{K}^0$ the set of non-Archimedean absolute values of $K$,  

$M_{K}^\infty$ the set of Archimedean absolute values of $K$,

$K_v$ the completion of $K$ at $v$,

$n_v$ the local degree $[K_v:\Q_v]$,

$G_\Q$ the Galois group $\mbox{Gal}(\overline \Q /\Q)$.

%***************************************************************

%************************************************** 
\newpage
\noindent \textbf{Part 1: Heights on the projective space} \\

In this lecture we define different heights of a point on a projective space, culminating in the Northcott and Kronecker theorems.  We discuss Segre's map and the d-uple  embedding with the intention of establishing a bound for the height of a point after a change of coordinates. 

%**********************************
\section{Heights on the projective space}
In this section we define the heights on the projective space over a number field $K$ and  give some basic properties of the heights function. 

Let $K$ be an algebraic number field and $[K:\Q]=n$. With $M_K$ we will denote the set of all absolute values in $K$.   For $v \in M_K$, the \textbf{ local degree at $v$}, denoted $n_v$ is 
\[n_v =[K_v:\Q_v]\]
where $K_v, \Q_v$ are the completions with respect to $v$. The following are true for any number field $K$, see \cite[pg. 171-172]{silv-book} for proofs.

i) \textbf{ Degree formula}. Let $L/K$ be an extension of number fields, and let $v \in M_K$ be an absolute value
on $K$. Then
\[\sum_{\substack{w \in M_L\\w|v}}[L_w:K_v]= [L:K]\]
ii) \textbf{Product formula}. Let $K$ be a number field, and let $x \in K^\star$. Then we say that $M_K$ satisfies the product formula if 
\[\prod_{v\in M_K}|x|_v=1\]\label{prod.formula}
Throughout this paper with $\overline \Q$ we will denote  the algebraic closure of $\Q$ and with $G_\Q:=\mbox{Gal}(\overline \Q /\Q)$. 

Given a point $P \in \P^n(\overline \Q)$ with homogenous coordinates $[x_0, \dots, x_n]$, the field of definition of $P$ is 
\[\Q(P)=\Q(x_o/x_j, \dots , x_n/x_j)\] 
for any $j$ such that $x_j \neq 0$.  

Let $K$ be a number field,  $\P^n(K)$ the projective space, and  $P \in \P^n(K)$ a point with homogenous coordinates $P=[x_0, \dots , x_n]$, for  $x_i \in K$. The \textbf{multiplicative heigh}t of $P$ is defined as follows
\[H_K(P) := \prod_{v \in M_K} \max\left\{\frac{}{}|x_0|_v^{n_v} , \dots, |x_n|_v^{n_v}\right \}\]
The \textbf{logarithmic height} of the point $P$ is defined as follows
\[h_K(P) := \log H_K(P)=   \sum_{v \in M_K}   \max_{0 \leq j \leq n}\left\{\frac{}{}n_v \cdot  \log  |x_j|_v \right\}.\]
\begin{exa} Let $P=[x_0, \dots , x_n] \in \P^n(\Q)$.  It is clear that $P$ will have a representative $[y_0, \dots, y_n]$ such that $y_i \in \Z$ for all $i$ and $\gcd(y_0, \dots, y_n)=1$.  With such representative for the coordinates of $P$, the non-Archimedean absolute values give no contribution to the height, and we obtain
\[H_\Q(P)=  \max_{0 \leq j \leq n}\left\{\frac{}{}|x_j|_\infty \right\}\]
\end{exa}
Next we will give some basic properties of heights functions.
\begin{lem}Let $K$ be a number field and $P\in \P^n(K)$. Then the following are true:

i) The height $H_K(P)$ is well defined, in other words it does not depend on the choice of homogenous coordinates of $P$

ii) $H_K(P) \geq 1$.
\end{lem}

\proof  
i)  Let $P=[x_0, \dots , x_n] \in \P^n(K)$. Since $P$ is a point in the projective space, any other choice of homogenous coordinates for $P$ has the form $[\lambda x_0, \dots, \lambda x_n]$, where  $ \lambda \in K^*$. Then
\[
\begin{split}
H_K\left([\lambda x_0, \dots, \lambda x_n]\right) &= \prod_{v \in M_K} \max_{0 \leq i \leq n}  \left\{\frac{}{} |\lambda x_i|_v^{n_v}\right\} =  \prod_{v \in M_K} |\lambda |_v^{n_v} \max_{0 \leq i \leq n}  \left\{\frac{}{} | x_i|_v^{n_v}\right\}\\
& =\left(   \prod_{v \in M_K} |\lambda |_v^{n_v}\right)\cdot \left(  \prod_{v \in M_K} \max_{0 \leq i \leq n}  \left\{\frac{}{} |x_i|_v^{n_v}\right\} \right)\\
%&= \prod_{v \in M_K} \max_{0 \leq i \leq n} \{ |x_i|_v^{n_v} \}
\end{split} 
\]
Applying the product formula we have
\[H_K\left([\lambda x_0, \dots, \lambda x_n]\right) = \prod_{v \in M_K} \max_{0 \leq i \leq n} \left\{\frac{}{} |x_i|_v^{n_v}\right\}= H_K(P)\]
And this completes the proof of the first part.

ii) For every point $P \in \P^n(K)$ we can find a representative of $P$ with homogenous coordinates such that one of the coordinates is 1. Let us reorder the coordinates of $P= [1, x_1, \dots, x_n]$ and calculate the height.
\[\begin{split} 
H_K(P) &=   \prod_{v \in M_K} \max\left\{\frac{}{}|x_0|_v^{n_v} , \dots, |x_n|_v^{n_v} \right\} =   \prod_{v \in M_K}  \max\left\{\frac{}{} 1, |x_1|_v^{n_v} , \dots, |x_n|_v^{n_v}\right\}  
\end{split}
 \] 
Hence,  every factor in the product is at least 1.  Therefore, $H_K(P) \geq 1$. 
\qed
 
\begin{lem}\label{lem_1}
Let $P \in \P^n(K)$ and $L/K$ be a finite extension. Then,
\[H_L(P)=H_K(P)^{[L:K]}.\]
\end{lem}

\begin{proof}
 Let $L$ be a finite extension of $K$ and $M_L$ the corresponding  set of absolute values. Then,  
\[\begin{split}
H_L(P) &= \prod_{w \in M_L} \max_{0 \leq i \leq n}\left\{\frac{}{} |x_i|_w^{n_w} \right\}  =  \prod_{v \in M_K} \prod_{\substack{w \in M_L\\  w|v}}   \max_{0 \leq i \leq n}\left\{\frac{}{} |x_i|_v^{n_w} \right\}, \qquad  \text{ since $x_i \in K$ } \\
& = \prod_{v \in M_K}   \max_{0 \leq i \leq n}\left\{\frac{}{} |x_i|_v^{n_v \cdot [L:K] }\right \},\qquad \text{ (product formula)}\\
&= \prod_{v \in M_K}   \max_{0 \leq i \leq n}\left\{\frac{}{} |x_i|_v^{ n_v} \right\}^{[L:K]}=H_K(P)^{[L:K]}
\end{split} \]
This completes the proof. 
\end{proof}

%\subsection{Absolute heights}
Using Lemma \ref{lem_1}, part ii),  we can define the height on $\P^n(\overline \Q)$. The height of a point on $\P^n(\overline \Q)$ is called the \textbf{absolute (multiplicative) height} and is the function 
\[
\begin{split}
H: \P^n(\bar \Q) & \to [1, \infty)\\
H(P)&=H_K(P)^{1/[K:\Q]},
\end{split}
\]
where  $P \in \P^n(K)$, for any $K$.  The \textbf{absolute (logarithmic) height} on $\P^n(\overline \Q)$  is the function 
\[
\begin{split}
h: \P^n(\bar \Q) & \to [0, \infty)\\
h(P)&= \log H(P)= \frac 1 {[K:\Q]}h_K(P).
\end{split}
\]
\begin{exa} Let $\a \in K$ be an algebraic number. The \textbf{height} of $\a \in K$ is the height of the corresponding projective point $(\a, 1) \in \P^1(K)$. Thus,
\[H_K(\a)= \prod_{v \in M_K} \max \left\{\frac{}{} 1, |\a|_v^{n_v}\right\}\]
and similarly for $h_K(\a), H(\a), h(\a)$.
\end{exa}

\begin{lem}\label{lem_galois_conj}
The height is invariant under Galois conjugation. In other words, for  $P \in \P^n(\overline \Q)$ and $\sigma \in G_{ \Q}$ we have $H(P^\sigma) = H(P)$.
\end{lem}

\proof   Let $P =[x_0, \dots, x_n] \in \P^n(\overline \Q)$. Let $K$ be a finite Galois extension of $\Q$ such that $P \in \P^n(K)$. Let $\sigma \in G_\Q$. Then $\sigma$ gives an isomorphism 
\[\sigma: K \to K^\sigma\]
and also identifies the sets $M_K$, and $M_{K^\sigma}$ as follows
\[
\begin{split}
\sigma: M_K &\to M_{K^\sigma}\\
v &\to v^\sigma
\end{split}
\]
Hence, for every $x \in K$ and $v \in M_K$, we have $|x^\sigma|_{v^\sigma} = |x|_v$.   Obviously $\sigma$ gives as well an isomorphism 
\[\sigma: K_v \to K^\sigma_{v^\sigma}\]
Therefore $n_v= n_{v^\sigma}$, where $n_{v^\sigma} = [ K^\sigma_{v^\sigma}: \Q_v]$. Then 
\[
\begin{split}
H_{K^\sigma}(P^\sigma) &= \prod_{w \in M_{K^\sigma}} \max_{0 \leq i \leq n} \left\{\frac{}{}  |x_i^\sigma|_{w}^{n_{w}}\right\}\\
&=  \prod_{v \in M_{K}} \max_{0 \leq i \leq n} \left\{\frac{}{}  |x_i^\sigma|_{v^\sigma}^{n_{v^\sigma}}\right\}=  \prod_{v \in M_{K}} \max_{0 \leq i \leq n} \left\{\frac{}{}  |x_i|_v^{n_v}\right\}=H_K(P)
\end{split}
\]
This completes the proof.
\qed
 
The following is known in the literature as Northcott's theorem.

\begin{thm}[Northcott]  \label{thm_finite}
Let $c_0$ and $d_0$ be constants. Then the set 
\[\{P \in \P^n(\overline \Q): H(P) \leq c_0 \text{ and } [\Q(P):\Q] \leq d_0\}\]
%note: \Q(P) is the minimal field of definition for P
contains only finitely many points. In particular for any number field $K$
\[\{P \in \P^n(K): H_K(P) \leq c_0 \} \]
is a finite set.
\end{thm}

\proof   Let $P=[x_0, \dots, x_n] \in \P^n(\overline \Q)$ be a point such that some $x_{i_0}=1$. Then for any absolute value $v$, and for all $0 \leq i \leq n$ we have
\[\max\left\{\frac{}{}|x_0|_v^{n_v} , \dots, |x_n|_v^{n_v}\right \} \geq \max\left\{\frac{}{}1, |x_i|_v^{n_v}\right\} . \]  
Let $\Q(P)$ be the  field of definition of $P$. Let us first estimate  $H_{\Q(P)}(P)$.
\[
\begin{split}
H_{\Q(P)}(P)&= \prod_{v \in M_{\Q(P)}} \max\left\{\frac{}{}|x_0|_v^{n_v} , \dots, |x_n|_v^{n_v} \right\} \\
 & \geq  \prod_{v \in M_{\Q(P)}} \max\left\{\frac{}{}1, |x_i|_v^{n_v}  \right\}, \, \, \text{ for all $0 \leq i \leq n$. }\\
 & = H_{\Q(P)}(x_i), \, \, \text{ for all $0 \leq i \leq n$. }  \\
\end{split}
\]
Taking the $[\Q(P):\Q]$-th root we have $ H(x_i) \leq H(P)$,  for all $0 \leq i \leq n$.  Clearly, $\Q(x_i) \subset Q(P)$, for all $0 \leq i \leq n$ and therefore $[\Q(x_i):\Q] \leq [\Q(P):\Q]$.  Then for all $0 \leq i \leq n$ we have,  
\[ H(x_i) \leq c_0 \text{ and } [\Q(x_i):\Q] \leq d_0.\] 
It suffices to show that for each $1\leq d\leq d_0$ the set 
\[\{x \in \overline \Q: H(x) \leq c_0 \text{ and } [\Q(x):\Q] = d\}\]
is finite (i.e we are considering the case when $n=1$). 

Assume, for some $x\in \overline \Q$, we have $[\Q(x):\Q]=d$.  Let $x_1, \dots, x_d$ be the $d$ conjugates of $x$ in $\Q$.  Then the minimal polynomial of $x$ over $\Q$ is 
\[f_x(t)= \min(x, \Q, t)= \prod_{j=1}^d(t-x_j)=\sum_{r=0}^d(-1)^rs_r(x)t^{d-r}.\]
Then for any absolute value $v\in M_{\Q(x)}$ we have
\[
\begin{split}
|s_r(x)|_v &=\left| \sum_{1\leq i_1 \leq \dots \leq i_r \leq d} x_{i_1}  \cdots x_{i_r}  \right|_v  \leq |c(r, d)|_v \max_{1\leq i_1 \leq \dots \leq i_r \leq d} \left\{ \frac{}{}|x_{i_1}  \cdots x_{i_r}|_v \right\}  \leq \\
&\leq |c(r, d)|_v \max_{1\leq i \leq d}\left\{ \frac{}{}|x_i|_v^r \right\} \leq |c(r, d)|_v \prod_{i =1}^{d}\left\{ \frac{}{}|x_i|_v \right\}^r 
%??????
\end{split}
\]
Where, $c(r, d)$ represents the number of terms in a symmetric polynomial with degree $r$ and $d$ variables, and is $\binom{d}{r}$. Then, 
\[
|c(r, d)|_v= \left \{ 
\begin{split}
&\binom{d}{r}  \qquad \, \, \, \text{ if  $v$ is Archimedean }\\
&  \,\,\,\, \,  1 \hspace{12mm} \text{ if  $v$ in non-Archimedean} 
\end{split}
\right.
\]
Hence, $c(r, d) = \binom{d}{r} \leq 2^d$  when $v$ is Archimedean, and 1 if $v$ in non-Archimedean. Now let us take the maximum over all symmetric polynomials. We have
\[
\begin{split}
\max \left\{ \frac{}{}|s_0(x)|_v, \dots, |s_d(x)|_v\right\} & \leq |s_i(x)|_v, \qquad \text{ (for some $1 \leq i \leq d$)}\\
%&\leq  \prod_{i=1}^d |2|_v^d \max_{1\leq i \leq d}\left\{ \frac{}{}|x_i|_v^d\right\}\\
&\leq|c( d)|_v \prod_{i=1}^d\max\left\{ \frac{}{}1, |x_i|_v\right\}^d,\\
\end{split}
\]
where, as above $|c( d)|_v = \binom{d}{r}$ when $v$ is Archimedean and 1 otherwise. Now we can calculate the height of $(s_0(x), \dots, s_d(x))$.
\[
\begin{split} 
H_{\Q(x)} (s_0(x), \dots, s_d(x))    &= \prod_{v\in M_\Q(x)}\max_{0 \leq i \leq d}\left\{ \frac{}{}|s_i(x)|_v^{n_v}\right\} \leq \prod_{v\in M_\Q(x)} |c(d)|_v^{ n_v } \prod_{i=1}^d\max_{ i}\left\{ \frac{}{}|x_i|_v^{n_v}, 1\right\}^d\\
%&\leq 2^{d \cdot [\Q(x):\Q]} \prod_{i=1}^d H_{\Q(x)}  (x_i)^d
\end{split}
\]
Using the degree formula 
\[ \prod_{v \in M_{\Q(x)}} |c(d)|_v^{n_v}=  \prod_{v \in M_{\Q(x)}^\infty} |c(d)|_v^{n_v} = c(d)^{[\Q(x):\Q]}\leq 2^{d^2}\]
we have
\[H_{\Q(x)} (s_0(x), \dots, s_d(x))\leq 2^{d^2} \prod_{i=1}^d H_{\Q(x)}  (x_i)^d\]
Taking, $[\Q(x):\Q]$-th root of both sides we have
\[H(s_0(x), \dots, s_d(x)) \leq 2^d \prod_{i=1}^d H  (x_i)^d\]
But the $x_i$'s are conjugates and by Lemma~\ref{lem_galois_conj} they all have the same height. Hence,
\[
\begin{split} 
H(s_0(x), \dots, s_d(x))&\leq 2^d H(x)^{d^2} \leq (2c_0^d)^d \, \, \,\,\, \text{ since $H(x) \leq c_0$ }
\end{split}
\]
Since the $s_i$'s are in $\Q$, is clear that for a given $c$ and $d$ there are only finitely many possibilities for the polynomial $f_x(t)$, and therefore only finitely many possibilities for $x$. Hence the set
\[\{x \in \overline \Q: H(x) \leq c_0 \text{ and } [\Q(x):\Q] = d\}\]
is finite. 
\qed

\begin{lem}[Kronecker's theorem]  Let $K$ be a number field,  and let $P=[x_0, \dots , x_n] \in \P^n(K)$. Fix any $i_0$ with $x_{i_0} \neq 0$. Then $H(P)= 1$ if and only if the ratio $x_j/x_{i_0}$ is a root of unity or zero for every $0 \leq j \leq n$. 
\end{lem}

\proof   Let  $P=[x_0, \dots , x_n] \in \P^n(K)$. Without loss of generality we can divide the coordinates of $P$ by $x_{i_0}$ and then reorder them. Assume, $P=[1, y_1, \dots, y_n]$ where $y_1, \dots, y_n$ are of the form $x_j/x_{i_0}$. If $y_l$   is a root of unity for every $1 \leq l \leq n$ then $|y_l|_v=1$ for every $v \in M_K$. Hence, $H(P)=1$.

Assume $H(P)=1$. Let $P^r=[x_0^r, \dots, x_n^r]$, for $r=1, 2, 3 \dots$. Then, from the definition of the height is clear that $H(P^r)=H(P)^r$, for every $r\geq 1$.  But $P^r \in \P^n(K)$ and by Theorem~\ref{thm_finite} we have that
\[\{P^r \in \P^n(K): H_K(P^r) \leq c\} \]
is a finite set. In this case $c=1$ and therefore  the sequence $P, P^2, P^3, \dots$ contains only finitely many distinct points. Choose integers $s>r\geq 1$ such that $P^s=P^r$. This implies that for each $1\leq j\leq n$ we have $x_j^s=x_j^r$. Therefore, $x_j^{s-r}=1$, where $s-r >0$. Therefore, each $x_j$ is a root of unity or is zero. 
\qed

%************************************************ 
\section{Segre map and $d$-uple embedding}

Let $m, n \geq 1$ and let $N= (n+1)(m+1)-1$. The \textbf{Segre map} is the map 
\[
\begin{split}
S_{n, m} \, : \, \,  \P^n(\bQ) \times \P^m(\bQ)   &\to  \P^N(\bQ) \\
 (P, Q) &\to  [x_0y_0, x_0y_1, \dots, x_iy_j, \dots, x_ny_m] 
\end{split}
\]
where $P= [x_0, \dots, x_n] \in \P^n(\bQ) $ and $Q=[y_0, \dots, y_m] \in \P^m(\bQ)$. The Segre maps are morphisms and give embeddings of the product $\P^n(\bQ) \times \P^m(\bQ)$ into $\P^N(\bQ)$. 
Next we will see how some of the properties of the heights are carried over through Segre embeddings.

\begin{lem}\label{exa_segre}
Let $S_{n,m}$ be the Segre embedding, $P \in \P^n(\bQ)$, and  $Q \in \P^m(\bQ)$. Then,
\[H(S_{n,m}(P,Q))=H(P) \times H(Q).\]
\end{lem}

\proof
Let $K$ be some number field such that  $P \in \P^n(K)$, and $Q \in \P^m(K)$, and $R=[z_0, \dots ,z_N]=S_{n,m}(P,Q)\in\P^N(K)$. For every absolute value $v \in M_K$ the following is true
\[
\begin{split}
\max_{0 \leq l \leq N}\left\{\frac{}{}|z_l|_v \right\}&=\max_{\substack{0 \leq i \leq n\\ 0 \leq j \leq m}} \left\{\frac{}{}|x_iy_j|_v\right\} \hspace{6mm} \text{ (by definition of Segre map)}\\
&=\max_{\substack{0 \leq i \leq n\\ 0 \leq j \leq m}}\left\{\frac{}{}|x_i|_v \cdot |y_j|_v \right\}   \text{ (by absolute value properties)}\\
&=\left( \max_{0 \leq i \leq n}\left\{\frac{}{}|x_i|_v \right\}\right) \cdot \left( \max_{0 \leq j \leq m}\left\{\frac{}{}|y_j|_v \right\}\right) 
\end{split}
\]
Let us calculate 
\[
\begin{split}
H_K(S_{n,m}(P,Q))&= \prod_{v \in M_K}\max_{0\leq l\leq N}\left\{\frac{}{}|z_l|_v^{n_v}\right\}= \prod_{v \in M_K} \left( \max_{0 \leq i \leq n} \left\{\frac{}{}|x_i|_v^{n_v} \right\}\right) \cdot \left( \max_{0 \leq j \leq m}\left\{\frac{}{}|y_j|_v^{n_v}\right\} \right)\\
&= \prod_{v \in M_K} \left( \max_{0 \leq i \leq n}\left\{\frac{}{}|x_i|_v^{n_v} \right\}\right) \cdot \prod_{v \in M_K} \left( \max_{0 \leq j \leq m}\left\{\frac{}{}|y_j|_v^{n_v}\right\} \right)= H_K(P)\cdot H_K(Q)
\end{split}
\]
Taking $[K:\Q]$-root  of both sides we obtain the desired result.
\qed

%\subsection{ $d$-uple embedding}

Let $P=[x_0, \dots, x_n] \in \P^n(\bQ)$. Let $M_0(x), \dots, M_N(x)$ be the complete collection of monomials of degree $d$ in the variable $x=(x_0, \dots, x_n)$. Note that $N$ is the number of monomials of degree $d$ in $n+1$ variables minus 1,   hence $N= \binom{n+d}{n}-1.$

Then, the map 
\[
\begin{split}
\phi_d: \P^n(\bQ) &\to \P^N(\bQ)\\
P &\to [M_0(x), \dots, M_N(x)]
\end{split}
 \]
is called the \textbf{ $d$-uple embedding}  of $\P^n(\bQ)$. This is a morphism, and in fact is an embedding of $\P^n(\bQ)$ into $\P^N(\bQ)$.  Next we describe a formula for the height under a $d$-uple embedding. 

\begin{lem} % combine described in Example A.1.2.6(a). and part iii) of B.2.4. 
Let $\phi_d: \P^n(\bQ) \to \P^N(\bQ)$ be the $d$-uple embedding.  Then   for all $ P \in \P^n(\overline \Q)$ we have
\[H(\phi_d(P))=  H(P)^d.\]
\end{lem}

\proof
Let $P$, and $\phi_d(P)=[ M_0(x), \dots, M_N(x)]$ be as above. By definition $M_i(x)$ are all monomials of degree $d$ in $n+1$ variables. It is clear that
\[|M_i(x)|_v \leq \max_i\left\{\frac{}{}|x_i|_v^d\right\}\]
and since $x_0^d, \dots, x_n^d$ appear in the list we have
\[\max_{0\leq j\leq N}\left\{ \frac{}{}|M_j(x)|_v \right\}= \max_{0\leq i \leq n}\left\{\frac{}{}|x_i|_v^d\right\}\]
Let $K$ be a number field such that $P\in \P^n(K)$, and $\phi_d(P)\in \P^m(K)$. Then, 
\[
\begin{split}
H_K(\phi_d(P))&= \prod_{v \in M_K}\max_{0\leq j\leq N}\left\{\frac{}{}|M_j(x)|_v^{n_v}\right\}= \prod_{v \in M_K}  \max_{0 \leq i \leq n}\left\{\frac{}{}|x_i|_v^{d \cdot n_v} \right\} \\
&= \left(\prod_{v \in M_K}  \max_{0 \leq i \leq n}\left\{\frac{}{}|x_i|_v^{ n_v} \right\}   \right)^d=H_K(P)^d
\end{split}
\]
Taking $[K:\Q]$-th root of both sides we obtain the desired result. 
\qed

For $P=[x_0, \dots, x_n]$ and $m \geq 1$, let $P^{(m)}$ be the point whose projective coordinates are all the monomials of degree $m$ in the $x_i$, and $P^m = [x_0^m, \dots, x_n^m]$.
Let $K$ be a number field such that $P^m \in \P^n(K)$. Then,
\[
\begin{split}
H_K(P^m) &= \prod_{v \in M_K} \max  \left\{\frac{}{} |x_0^m|_v^{n_v}, |x_1^m|_v^{n_v},  \dots, |x_n^m|_v^{n_v} \right\}= \prod_{v \in M_K} \max_i\left\{\frac{}{}|x_i^m|_v^{n_v} \right \}\\
&= \prod_{v \in M_K} \max_i\left\{\frac{}{}|x_i|_v^{n_v} \right \}^m =H_K(P)^m
\end{split}
\]
Then,   $H(P^{(m)})= H(P^m)=H(P)^m$. 

%\begin{defi}
%Let $\phi : X \to Y$ be a covering of varieties $X$ and $Y$.  The height of $\phi$, denoted by $H(\phi)$,  is defined as follows ....
%\end{defi}

%***************************************************************
\section{Heights and change of coordinates on $\P^n$}
In the next few paragraphs we will consider what happens to the height of a point after a transformation $\phi$.  
Let 
\[
\begin{split}
 \phi : \quad \qquad \P^n (K)  & \to \P^r (K) \\
  [x_0, \dots , x_n] & \to [ \phi_o, \dots , \phi_r] \\
  \end{split}
\] 
be a rational map such that $\phi_i$ are rational functions of degree $m$. Define \textbf{the height of the map $\phi$ }, denoted by $H(\phi)$,  to be the height of a point $P$ in the projective space, where $P$ is the sequence  of  coefficients of all the $\phi_i$'s.  

Denote by $\mathcal Z$ be the set of common zeroes for all $\phi_i$'s. Then $\phi$ is defined on $\P^n (\bQ) \setminus \mathcal Z$.  We have the following: 
\begin{lem}[Formula for changing coordinates]\label{ch-coord}  
The following are true:

i)  Let $\phi$ be as above, and $\phi_i$  homogenous polynomials of degree $m$.  Then for each point $P=[x_0, \dots , x_n] \in \P^n (\bQ) \setminus \mathcal Z$ we have 
\[H(\phi \left( P  \right) ) \leq ||N||_\infty \, H(\phi) \, H(P )^m\]
where  $N$ is the maximum number of monomials appearing in any one of the $\phi_i$, and 
\[||N||_\infty = \prod_{v \in M_K^\infty}|N|_v^{n_v}\]

ii) Let $X$ be a closed subvariety of $\P^n (\bQ)$ with the property that $X\cap \mathcal Z= \emptyset$. Thus $\phi$ defines a morphism $X \to \P^r (\bQ)$. Then for every 
$P=[x_0, \dots , x_n]  \in X$ we have 
\[ H ( \phi ( P)) = c_0 \cdot H(P)^m .\]
for some constant $c_0$. 
\end{lem}

\proof Fix a field of definition $K$ for $\phi$, so $\phi_0, \dots, \phi_r  \in K[X_0, \dots, X_n]$.   We can write $\phi_i$'s as follows 
\[\phi_i(X)= \sum_{\substack{j=(j_0, \dots, j_n)\in I\\j_0+\dots+j_n=m}}a_{i_j}X^j \quad \text{ for all $0 \leq i \leq r$} \]
where $X= X_0 X_1 \cdots  X_n$ and $X^j=X_0^{j_0}\cdot X_1^{j_1} \cdots X_n^{j_n}$. For some  $P=[x_0, \dots, x_n]$,  we want to estimate $H(\phi(P))$ where $\phi(P) =(\phi_0(P), \dots, \phi_r(P))$. 
\[
\begin{split}
H_K(\phi(P)) &=\prod_{v \in M_K} \max \left\{ \frac{}{} |\phi_0(P)|_v^{n_v}, \dots, |\phi_r(P)|_v^{n_v}\right\}= \prod_{v \in M_K} \max_{0\leq i \leq r} \left\{\frac{}{} |\phi_i(P)|_v^{n_v}\right\} \\
&= \prod_{v \in M_K} \max_{0\leq i \leq r} \left\{ \left|\sum_{\substack{j=(j_0, \dots, j_n)\in I\\j_0+\dots+j_n=m}}a_{i_j}x_0^{j_0}\cdot x_1^{j_1} \cdots x_n^{j_n}\right|_v^{n_v}\right\}\\
& \leq \prod_{v \in M_K} N_v^{n_v} \cdot \max_{\substack{i, j_l\\0\leq l \leq n}} \left \{\left|\frac{}{}a_{i_{j_l}} x_0^{j_0}\cdot x_1^{j_1} \cdots x_n^{j_n}\right|_v^{n_v} \right\}\\
& \leq \prod_{v \in M_K} N_v^{n_v} \cdot \max_{\substack{i, j_l\\ 0 \leq l \leq n}} \left \{|a_{i_{j_l}} |_v^{n_v} \right\} \cdot \max_{ 0 \leq l \leq n} \left \{| \frac{}{}x_l  |_v^{n_v} \right\}^m\\
&=   ||N||_{\infty} \cdot  H_K(\phi) \cdot \,  H_K(P)^m
\end{split}
\]
where $N$  is the maximum  number of monomials appearing in any one of the $\phi_i$.  Taking $[K:\Q]$-th root of both sides we obtain the desired result. 

ii) In part (i) we proved  that 
\[H(\phi(P)) \leq c_1 \cdot H(P)^m\]
where $c_1=  ||N||_\infty \cdot H(\phi)$, and it depends on $\phi$ but does not depend on the point $P \in \P^n(\bQ)$.  Now we want to show that for a point $P=[x_0, \dots, x_n] \in X(K)$ and a morphism $\phi=(\phi_0, \dots, \phi_r)$ on $X$ the following holds
\[H(\phi(P)) \geq  c_2 \cdot H(P)^m\]
Let $f_1, \dots, f_l$ be homogenous polynomials generating the ideal of $X$. Then, $f_1, \dots, f_l$, $\phi_0, \dots, \phi_r$ have no common zeros in $\P^n$. Let $ \mathfrak J =\< f_1, \dots, f_l, \phi_0, \dots, \phi_r \rangle$ and $  \mathfrak  I= \< X_0, \dots, X_n \rangle$. From Nullstellensatz theorem we have that $\mathfrak J$ has a radical equal to $ \mathfrak I$. Hence, for some polynomials  $p_{i, j}, q_{i, j}$ and an exponent $t \geq m$ the following is true
\[p_{0,j} \phi_0 + \cdots + p_{r, j}\phi_r+ q_{1,j}f_1+ \cdots +q_{l,j}f_l =X_j^t \qquad \text{ for } 0\leq j \leq n   \]
Note that, since $\phi_i$'s have degree $m$ then $p_{i, j}$'s have degree $t-m$. Extending $K$ if necessary we can assume that $p_{i, j}$'s, and $q_{i, j}$'s have coefficients in $K$. Since $P \in X(K)$, then   $f_i(P)=0$, for all $0 \leq i \leq l$. Evaluating the above at the point $P$ we have
\[p_{0,j}(P) \phi_0(P) + \cdots + p_{r, j}(P)\phi_r(P)= x_j^t, \qquad 0 \leq j \leq n\]
Hence,
\[
\begin{split}
|P|_v^t&= \max_j \left\{ \frac{}{}|x_j|_v^t\right\}= \max_j \left\{ \frac{}{}|p_{0,j}(P) \phi_0(P) + \cdots + p_{r, j}(P)\phi_r(P)|_v \right\}\\
&\leq |r+1|_v \left(\max_{i,j} \left\{ \frac{}{}|p_{i,j}(P)|_v \right\} \right)\left(\max_i \left\{ \frac{}{}|\phi_i(P)|_x\right\}\right)\\
& \leq |r+1|_v \left(\left| \binom{t-m+n}{n}\right|_v |P|_v^{t-m} \max_{i,j} \left\{ \frac{}{}|p_{i,j}|_v \right\}  \right) \left(\max_i \left\{ \frac{}{}|\phi_i(P)|_x\right\}\right)\\
\end{split}
\]
Denoting by $c_2$ the following
\[c_2 = |r+1|_v \cdot \left| \binom{t-m+n}{n}\right|_v \cdot \max_{i,j} \left\{ \frac{}{}|p_{i,j}|_v \right\}\]
and multiplying the above over all $v \in M_K$ and then taking $n_v/ [K:\Q]$-th root we obtain
\[H(P)^t \leq c_2 \cdot H(P)^{t-m}H(\phi(P)).\]
This completes the proof.
\qed

\begin{rem} If the change of coordinates is done by an automorphism of $\P^n(K)$, say $M \in PGL_{n+1}(K)$, then
\[H( P^M )  \leq (n+1) \cdot H(M)  \cdot H(P), \]
where $H (M)$ is 
\[ H(M) =  \max \{ a_{i, j} \},\]
for $1 \leq i \leq n+1$ and $1 \leq j \leq n+1$. 
\end{rem}

% ****************************************************
\newpage
\noindent \textbf{Part 2: Heights of polynomials} \\

In this lecture  we   define the height of a polynomial.    This is interesting to us since in the next section we will define the height of algebraic curves in terms of the height of a polynomial.

\section{Heights of polynomials}\label{heights_pol}

Throughout this paper a non-homogenous polynomial with $n$ variables will be denoted as follows
\[f(x_1, \dots, x_n) = \sum_{\substack{i=(i_1, \dots, i_n) \in I} } a_ix_1^{i_1} \cdots x_n^{i_n}\]
where all $a_i \in K$,  $I \subset \Z^{\geq 0}$, and $I$ is finite. Let $\deg f$ denote the total degree of $f$. We will use lexicographic ordering to order the terms in a given polynomial, and $x_1 > x_2 > \dots >x_n$.

The \textbf{(affine) multiplicative height of $f$} is defined as follows 
\[H^{\mathbb A}_K(f)= \prod_{v \in M_K} \max\left\{\frac{}{}1, |f|_v^{n_v}\right\}\]
where 
\[|f|_v := \max_j\left\{\frac{}{} |a_j|_v \right\}\]
is called the \textbf{Gauss norm} for any absolute value $v$.  The \textbf{(affine)  logarithmic height of $f$} is defined to be
\[h^{\mathbb A}_K(f) = h_K([1, \dots, a_j, \dots]_{j \in I}).\]
Hence, the affine height of a polynomial is defined to be the height of its coefficients taken as affine coordinates.  While, the \textbf{(projective) multiplicative height of a polynomial} is the height of its coefficients taken as coordinates in the projective space. Thus,
\[H_K(f)= \prod_{v \in M_K}  |f|_v^{n_v}\]
and the \textbf{(projective) logarithmic  height} is 
\[h_K(f)=   \sum_{v\in M_K} n_v \log|f|_v\]
The   \textbf{(projective) absolute multiplicative  height} is defined as follows
\[
\begin{split}
 H: \P^n(\Q) & \to [1, \infty)\\
 H(f)&=H_K(f)^{1/[K:\Q]},
 \end{split}
\]
and in the same way $h(f)$, $H^{\mathbb A}(f), h^{\mathbb A}(f)$.

\begin{exa} Let 
\[f(x, y)= 3x^3+3x^2+ 12xy+ 6y^2+3y+6.\]
Since $f(x, y)$ has integer coefficients the non-Archimedean absolute values give no contribution to the height, 
the (affine) height is
\[
H^{\mathbb A}(3x^3+3x^2+12xy+6y^2+3y+6) =H^{\mathbb A}([1,3,3,12,6,3,6]) =  12. 
\]	
The (projective) height is
\[
\begin{split}
H(3x^3+3x^2+12xy+6y^2+3y+6)&=H([3,3,12,6,3,6])=H(1,1,4,2,1,2)= 4.
\end{split}
\]

\end{exa}

\begin{thm}\label{pol_finite} 
Let be given  $F(x, y ) \in K [x, y]$. Then,  there are only finitely many polynomials $G(x,y)\in K [x, y]$ such that $H_K(G) \leq H_K(F)$.  
\end{thm}

\proof
Let 
\[F(x, y)=\sum_{\substack{i=(i_1, i_2)\in I \\ i= i_1+i_2} } a_ix^{i_1}y^{i_2}\]
be a  polynomial with coefficients in $K$ and  fix an ordering $x > y$.  Let  $H_K(F)=c$. By  definition  
\[H_K(F) =  \prod_{v \in M_K}  |f|_v^{n_v}= \prod_{v \in M_K} \max_i\left\{\frac{}{} |a_i|_v^{n_v}\right\} = H_K[a_0   \dots, a_i, \dots ]_{i\in I}.\]
But, $P=[a_0   \dots, a_i, \dots ]_{i\in I}$ is a point in $\P^{s}$ where $s$ is the number of monomials of degree $d$ in 2 variables. Hence,  $s=\left( 
\begin{array}{c}
d+1 \\ 
d
\end{array}
\right). $

From Theorem~\ref{thm_finite} we have  that  for any constant  $c$ the set 
\[\{P \in \P^s(K): H_K(P) \leq c \}\]
is finite. Hence there are finitely many polynomials $G(x, y)$ with content 1 corresponding to points $P$ with height   
$H_K(G) \leq c=H_K(F)$.
\qed

Now we will study the height of the product of polynomials. At first we will deal with the case when the polynomials are in different variables, and then consider the case when they are polynomials in the same variable.

\begin{prop}Let $f(x_0, \dots, x_n)$ and $g(y_0, \dots, y_n)$ be polynomials in different variables. Then, the projective height has the following property
\[H(f \cdot g)= H(f) \cdot H(g)\]
\end{prop}

\proof  The height of a polynomial is equal to the height of its coefficients in appropriate projective space. Let $H(f) = H(P)$, where $P \in  \P^{s}$, and $H(g)=H(Q)$ for $Q \in \P^{l}$, where $s,l$ is  the number of monomials of $f, g$ respectively.  Then, $H(f \cdot g)= H(S_{s,l}(P, Q))= H(P) \cdot H(Q)$ from  Lemma~\ref{exa_segre}. Therefore,  $H(f \cdot g) = H(f) \cdot H(g)$.
\qed

Before considering the height of polynomials in the same variables, we will consider $|f \cdot g|_v$. The following lemma is true for the product of a finite number of polynomials.

\begin{lem}[Gauss's lemma] Let $K$ be a number field and $f, g \in K[x_1, \dots, x_n]$.  If $v$ is not Archimedean, then $|f g|_v= |f|_v  |g|_v$. 
\end{lem}

The proof can be found in \cite[pg. 22]{bombieri}.

Gauss's lemma applies to all non-Archimedean absolute values but the Archimedean case is more complicated. An analogous Archimedean estimate is given by the following lemma. Gauss's lemma and the following are used to give an estimate of $H(f_1 f_2 \cdots f_r)$ in terms of $H(f_i)$ for $ 1\leq i \leq r$ and  $f_1, f_2, \dots , f_r \in  K[x_1, \dots, x_n]$.

\begin{lem}\label{nonarch}
Let $ f_1, f_2, \dots, f_r \in   \C[x_1, \dots, x_n]$. Denote by $f= f_1\cdots f_r$   and $d_i =\deg (f, x_i)$. Then, the following is true
\begin{equation}\label{eq1}
\prod_{i=1}^r|f_i|_v \leq e^{(d_1+\dots+d_n)}|f|_v.
\end{equation}
\end{lem}

The proof of this can be found in \cite[pg. 232]{silv-book} and uses the concept of Mahler measure which is defined as follows. 

Let  $f(x_1, \dots, x_n) \in \C[x_1, \dots, x_n]$  be a polynomial in $n$ variables.  The \textbf{Mahler measure} of this polynomial is defined as follows
\[M(f):= \exp \left( \int_{\mathbb T^n} \log \left|f(e^{i\theta_1}, \dots , e^{i\theta_n})\right| d\mu_1 \cdots d\mu_n \right)\]
where $\mathbb T$ is the unit circle $\{e^{i\theta}|0\leq \theta \leq 2\pi\}$ equipped with the standard measure $d\mu= (1/2\pi)d\theta$

One of the most important properties of the Mahler measure is the multiplicative property. 
\[M(fg)= M(f)M(g),\] 
see \cite[pg. 230]{silv-book} for proof.

%\subsection{Non-homogenous polynomials}

The following is true for affine heights. 

\begin{lem}~\label{convers_gelfands}
Let $K$ be a number field and  $ f_1, \dots, f_r  \in K[x_1, \dots, x_n]$.  Denote with $\deg f_j$ the total degree of $f_j$. Then the following are true

i) The height of the product of $f_1, \cdots, f_r$ is bounded as follows
\[
\begin{split}
H^{\mathbb A}(f_1\, f_2 \cdots f_r) & \leq  \, \, N  \cdot  \prod_{j=1}^r\frac{}{}H^{\mathbb A}(f_j)  \\
%& \leq r \cdot \max_{1 \leq j \leq r}\left\{\frac{}{} h(f_j)+ (\deg f_j+m) \log2\right\}.  
\end{split}
\]

ii)The height of the sum of $f_1 + \cdots + f_r$ is bounded as 
\[H^{\mathbb A}(f_1 +f_2 + \dots + f_r) \leq \, \,   r \, \cdot  \prod_{j=1}^r H^{\mathbb A}(f_j).\]

iii)  Suppose that $f_1, \dots, f_r \in \O_K[x_1, \dots, x_n]$ have coefficients in the ring of integers $\O_K$ of $K$. Then
\[H^{\mathbb A}(f_1 +f_2 + \dots +f_r) \leq \, \, r   \cdot \max_j \left\{\frac{}{} H^{\mathbb A}(f_j)\right\}^{[K:\Q]}.\]
This estimate is useful when $K$ is fixed and $r$ is large. 
\end{lem}

\proof

i) Let $i=(i_1, \dots, i_n)$ and  write $f_j$'s as follows
\[f_j =\sum_ia_{ji} x_1^{i_1} \cdots x_n^{i_n}\]
for all $j=1, \dots, r$. Then
\[
\begin{split}
f_1 f_2 \cdots f_r &= \left(\sum_ia_{1i} x_1^{i_1} \cdots x_n^{i_n}\right)  \cdot  \left( \sum_ia_{2i} x_1^{i_1} \cdots x_n^{i_n} \right)\cdots  \left(\sum_ia_{ri} x_1^{i_1} \cdots x_n^{i_n}\right)\\
&=   \sum_i \left( \sum_{i_1+\dots +i_r =i} a_{1i_1}\cdot a_{2i_2} \cdots a_{ri_r}  \right) x^i
%?????????
\end{split}
\]
where   $x^i=x_1^{i_1} \cdots x_n^{i_n}$. Then, for every $v \in M_K$ the Gauss norm is
\[|f_1  f_2 \cdots f_r |_v =\max_i \left\{ \left| \sum_{i_1+\dots +i_r =i} a_{1i_1}\cdot a_{2i_2} \cdots a_{ri_r}  \right|_v\right\}.
\]
Let $N$ be an upper bound for the number of non-zero terms in the sums, and let 
\[N_v= \left \{ 
\begin{split}
& N  \, \, \, \text{ if $v$ is Archimedean }\\
&  1 \, \, \, \, \,  \text{ if $v$ is non-Archimedean } 
\end{split}
\right.
\]
Then,
\[
\begin{split}
|f_1  f_2 \cdots f_r |_v & \leq \max_i \left\{  \sum_{i_1+\dots +i_r =i} \left|a_{1i_1}\cdot a_{2i_2} \cdots a_{ri_r}  \right|_v\right\}\\
&\leq \max_i \left\{ N_v \cdot \max_{\substack{i_j\\i_1+\dots +i_r =i}} \left\{ \frac{}{}\left| a_{1i_1}\cdots a_{ri_r}  \right|_v\right\}\right\}\\
& \leq N_v \prod_{j=1}^r \max_{i_j}\left \{\frac{}{}1, |a_{ji_j}|_v\right\} \leq N_v \prod_{j=1}^r \max_{j}\left\{ \frac{}{}1, |f_j|_v \right\}.
\end{split}
\]
Raising to the $n_v$ power and taking the product over all valuations $v\in M_K$ we have the following
\[
\begin{split}
H^{\mathbb A}_K(f_1\cdots f_r ) &= \prod_{v \in M_K} \max \left\{\frac{}{} 1, |f_1\cdots f_r |_v^{n_v}\right\} \leq \prod_{v \in M_K} \left\{N_v \prod_{j=1}^r \max_j \left\{\frac{}{} 1, |f_j|_v\right\}\right\}^{n_v}\\
& \leq N^{[K:\Q]} \prod_{j=1}^r H_K(f_j), \qquad \left( \text{ since } \sum_{v \in M_K^\infty} n_v=[K:\Q]\right)
\end{split}
\]
Taking $[K:\Q]$-th root  we obtain the desired result.

ii)  Let $f_j$ be as above. Then,
\[f_1+ \dots + f_r = \sum_{(i_1, \dots, i_n)=i}(a_{1i}+\dots+a_{ri})x^i\] 
Thus, for every absolute value $v\in M_K$,
\[|f_1+ \dots + f_r|_v = \max_i \left\{ \frac{}{}|a_{1i}+\dots+a_{ri}|_v\right\}.\] 
Letting,
\[r_v= \left \{ 
\begin{split}
& r  \, \, \, \text{ if $v$ is Archimedean }\\
&  1 \, \, \,   \text{ if $v$ is non-Archimedean } 
\end{split}
\right.
\]
we have
\[
\begin{split}
 |f_1+ \dots + f_r|_v   &\leq r_v   \max_{j,i}\left\{\frac{}{}1, |a_{ji}|_v\right\} \quad \text{ (  for $j, i$  as above)}\\
& \leq r_v  \prod_{j=1}^r \max_i \left\{\frac{}{}1, |a_{ji}|_v\right\}.
\end{split}
\]
Raising to the $n_v/[K:\Q]$ power and taking the product over all valuations $v\in M_K$ we have the following
\[H^{\mathbb A}(f_1+ \dots + f_r) \leq \, \, r   \prod_{j=1}^r H^{\mathbb A}(f_j).\]
And we are done.

iii) We have that  $f_1, \dots, f_r $  have coefficients in the ring of integers $\O_K$ of $K$. Then, $f_1+ \dots + f_r$ will have integer coefficients as well. Hence, for any non-Archimedean absolute value $v$, and any $j$  we have that $|f_j|_v \leq 1$ and therefore the following is true
\[\max \left \{\frac{}{}1, |f_1+ \dots + f_r|_v \right \} = \max \left\{\frac{}{}1, |f_1|_v \right \} = \dots =  \max \left\{\frac{}{}1, |f_r|_v \right\}=1.\]
Hence the non-Archimedean absolute values do not contribute to $H_K(f_1+ \dots +f_r)$, and we have
\[ 
\begin{split}
H^{\mathbb A}_K(f_1+ \dots +f_r) & = \prod_{v \in M_k^\infty} \max \left\{\frac{}{}1, |f_1+ \dots + f_r|_v^{n_v}\right\}\\
&\leq  \prod_{v \in M_k^\infty}  r \cdot \max_{1\leq j \leq r} \left\{\frac{}{}1, |f_j|_v^{n_v}\right\}, \qquad \text{ from absolute value properties }\\
&\leq r^{[K:\Q]} \cdot \max_{1\leq j \leq r} \left\{ \max_{v\in M_K^\infty} \left \{\frac{}{}1, |f_j|_v^{n_v} \right\}^{[K:\Q]}\right\}, \qquad \text{ since $\#M_K^\infty \leq [K:\Q]$ }\\
& \leq r^{[K:\Q]} \cdot \max_{1\leq j \leq r} \left\{ H^{\mathbb A}_K(f_j)^{[K:\Q]}\right\}
\end{split}
\]
Taking $[K:\Q]$-th root of both sides we obtain the desired result.
\qed

The converse inequality for the inequality in part (i) is known as Gelfand's inequality. This inequality is true if we use projective polynomial heights.   

\begin{lem}[Gelfand's inequality]  Let $f_1, \dots, f_r \in \overline \Q[x_1, \dots, x_n]$ be polynomials, with degree $d_1, \dots, d_r$ respectively, such that $\deg (f_1 \cdots f_r, x_i) \leq d_i$ for each $1\leq i \leq r$. Then
\[\prod_{i=1}^r H(f_i) \leq e^{(d_i+ \cdots +d_n)} \cdot H(f_1 \cdots f_r) .\]
\end{lem}

\proof  From Lemma~\ref{nonarch} the following is true  
\begin{equation}\label{eq1}
\prod_{i=1}^r|f_i|_v \leq e^{(d_1+\dots+d_n)}|f|_v.
\end{equation}
Then, assuming the above we have
\[
\begin{split}
\prod_{i=1}^r H_K(f_i)&= \prod_{i=1}^r\prod_{v\in M_K} |f_i|_v^{n_v} = \prod_{v\in M_K} \prod_{i=1}^r |f_i|_v^{n_v} = \prod_{v\in M_K} \left(\frac{}{} |f_1|_v^{n_v} |f_2|_v^{n_v} \cdots |f_r|_v^{n_v}\right)\\
& \leq \prod_{v\in M_K^0}|f_1 \cdots f_r|_v^{n_v}\cdot \prod_{v\in M_K^\infty}e^{n_v (d_1+\dots+d_n)}|f_1 \cdots f_r|_v^{n_v}\\
& \leq e^{[K:\Q](d_1+\dots+d_n)}H_K(f_1 \cdots f_r). 
\end{split}
\]
Taking $[K:\Q]$-th root  of both sides we obtain Gelfand's inequality.  
\qed

\begin{lem} 
Let $K$ be a number field, $v$ an absolute value on $K$, and   $f \in K[x_1, \dots, x_n]$ a polynomial. Then,
\[ \left|\frac{\partial f}{\partial x_j}\right|_v \leq |\deg f|_v  \cdot  |f|_v.\]
\end{lem}

\proof  Let the polynomial $f$ be as follows
\[f(x_1, \dots, x_n) =   \sum_{\substack{i=(i_1, \dots, i_n) \in I} } a_ix_1^{i_1} \cdots x_n^{i_n}.\]
Then every coefficient of $ \partial f/\partial x_j$ has the form $c \cdot a_i$ for some positive integer $c \leq \deg f$ and some multi index $i$. Therefore,
\[
\begin{split}
 \left|\frac{\partial f}{\partial x_j}\right|_v &\leq \max_i \left\{ \max_{c \leq \deg f} \left\{ \frac{}{} |c \, a_i|_v \right\}\right\}= \left|\deg f\right|_v \cdot |f|_v.
\end{split}
\]
This completes the proof.
\qed

Let $b=(b_1, \dots, b_n) \in K^n$.   Denote with $ |b|_v = \max \left\{ \frac{}{}|b_i|_v \right\}.$

% ?????????????   what's this ?   Shouldn't this be the affine height?

\begin{lem}
Let $K$ be a number field,   $f \in K[x_1, \dots, x_n]$ a polynomial of degree $d$, and  $b=(b_1, \dots, b_n) \in K^n$.   Then, 
\[|f(b)|_v \leq \min \left\{ |2 d|_v^n, |2|_v^{d} \right\}  \cdot \max \{ 1, |b|_v\}^{d}\cdot |f|_v .\]
\end{lem}
The prof   can be found in \cite[pg. 236]{silv-book}.

Next we will consider bounds for the Gauss norm of a polynomial $f(x) \in K[x_1, \dots, x_n]$, first when we shift  $x=(x_1, \dots, x_n)$ with a vector $b=(b_1, \dots, b_n) \in K^n$, then when we multiply $x$ with $u=(u_1, \dots, u_n)$, and then when we  combine them. 

Let $b=(b_1, \dots, b_n) \in K^n$,  $|b|_v$ as  above, and  define a \textbf{ shifted polynomial} as follows
\[f_{b}(x)= f(x+b)= f(x_1+b_1, \dots, x_n+b_n).\]

\begin{lem} Let $K$ be a number field,   $f \in K[x_1, \dots, x_n]$ such that  $\deg f = d$. The following statements are true.

i)    Let $b=(b_1, \dots, b_n) \in K^n$ and   $|b|_v$ as  above. The  height of the   shifted polynomial $f_{b}(x)$   is bounded by
\begin{equation}\label{eq1}
| f_{b}(x) |_v \leq |2|_v^{2 d} \cdot \max\{1, |  b|_v \}^{d} \cdot  |f|_v.
\end{equation}

ii) Let $ u=(u_1, \dots, u_n)$ and define    $f_u(x)= f(  u \cdot   x) = f(u_1x_1, \dots, u_nx_n)$.
Then,
\[
|f_{u}(x)|_v    \leq   \max\{1,|u|_v\}^d \cdot |f|_v .
\]

iii) For $b$, and $ u$ as above define    $ f(  u  x +  b) = f(u_1x_1 +b_1, \dots, u_nx_n +b_n)$. 
Then, 
\[
|f(  u  x +  b)|_v  \leq |2|_v^{2\deg  f} \cdot  \max\{1,|u|_v\}^d  \cdot \max\{1, |b|_v \}^{d} \cdot  |f|_v.
\]
\end{lem}

\proof i) Let
\[f(x_1, \dots, x_n) =   \sum_{\substack{i=(i_1, \dots, i_n) \in I} } a_ix_1^{i_1} \cdots x_n^{i_n}.\]
and  compute
\[
\begin{split}
f_{  b}(x) &=  \sum_i a_i(  x+   b)^i\\
&=  \sum_i a_i\left( \sum_{j_1=0}^{i_1} \binom{i_1}{j_1}x_1^{j_1}b^{i_1-j_1}   \right) \cdots \left( \sum_{j_n=0}^{i_n} \binom{i_n}{j_n}x_n^{j_n} b^{i_n-j_n}   \right)\\
&=   \sum_{j_1=0}^{d_1} \cdots \sum_{j_n=0}^{d_n} \left( \sum_{\substack{i_1, \dots, i_n\\ j_l\leq i_l \leq d_l }}a_i\binom{i_1}{j_1} \cdots \binom{i_n}{j_n}  
  \times  b_1^{i_1-j_1}\cdots b_n^{i_n-j_n} \right)\times  x_1^{j_1} \cdots x_n^{j_n} 
\end{split}
\]
Then, for every $v \in M_K$ the  Gauss norm is
\[
\left|\frac{}{}f_{  b}(x)\right|_v =\max_{\substack{j_1, \dots, j_n \\ 0 \leq j_l \leq d_l}} \left| \sum_{\substack{i_1, \dots, i_n\\ j_l\leq i_l \leq d_l }}a_i\binom{i_1}{j_1} \cdots \binom{i_n}{j_n} b_1^{i_1-j_1}\cdots b_n^{i_n-j_n}\right|_v.
\]
If we denote by $N$ be  number of the terms in the last sum, then $N$ is at most   $\prod_{l=1}^n(d_l +1) \leq \prod_{l=1}^n 2^{d_l}=2^d$.
Estimate the binomial coefficients we have,
\[
%\begin{split}
\binom{i_1}{j_1} \cdots \binom{i_n}{j_n} \leq 2^{i_1} \cdots 2^{i_n} = 2^{i_1 + \dots + i_n} \leq 2^{d_1 \cdots d_n} =2^d
%\end{split}
\]
Letting
\[N_v= \left \{ 
\begin{split}
& N\leq 2^d  \, \, \, \text{ if $v$ is Archimedean }\\
&  1 \, \, \, \, \, \, \,  \qquad \text{ if $v$ is non-Archimedean } 
\end{split}
\right.
\]
and using the above estimates we have
\[
\begin{split}
\left|\frac{}{}f_{  b}(x)\right|_v &=\max_{\substack{j_1, \dots, j_n \\ 0 \leq j_l \leq d_l}} \left\{ \left| \sum_{\substack{i_1, \dots, i_n\\ j_l\leq i_l \leq d_l }}a_i\binom{i_1}{j_1} \cdots \binom{i_n}{j_n} b_1^{i_1-j_1}\cdots b_n^{i_n-j_n}\right|_v\right\}.\\
& \leq N_v \cdot \max_{i, j} \left\{ \frac{}{}1,  \left|  a_i \binom{i_1}{j_1} \cdots \binom{i_n}{j_n} b_1^{i_1-j_1}\cdots b_n^{i_n-j_n}\right|_v\right\}.\\
& \leq N_v \cdot  2_\infty^{d}  \cdot \max\{1, |b_1^{i_1-j_1}\cdots b_n^{i_n-j_n}|_v\} \cdot  \max \{ |a_i|_v \}\\
&\leq   2_\infty^{2d} \cdot \max\{1, |b_1|_v^{i_1-j_1}\} \cdots \max\{1, |b_n|_v^{i_n-j_n}\} \cdot \max \{ |a_i|_v \} \\
& \leq 2_\infty^{2d} \cdot \max\{1, |b_1|_v^{d}\} \cdots \max\{1, |b_n|_v^{d}\} \cdot \max \{ |a_i|_v \} \\
& =   2_\infty^{2d} \cdot \max\{1, |  b|_v^{d}\} \cdot  |f|_v.
\end{split}
\]
This completes the proof.

ii) Let us evaluate  
\[
\begin{split}
f_{  u}(  x)&=f(u_1 \cdot x_1, \dots, u_n \cdot x_n)  \\
& = \sum_{\substack{i=(i_1, \dots, i_n) \in I} } a_i (u_1 x_1)^{i_1} \cdots (u_n x_n)^{i_n}\\
&= \sum_{\substack{i=(i_1, \dots, i_n) \in I} } a_i \cdot ( u_1^{i_1} \cdots u_n^{i_n} )\cdot (x_1^{i_1} \cdots  x_n^{i_n})\\
\end{split}
\]
Then, for every $v \in M_K$ the  Gauss norm is
\[
\begin{split}
|f_{  u}(  x)|_v &= \max_i \left\{\frac{}{}| a_i   u_1^{i_1} \cdots u_n^{i_n}|_v \right\} \\
&\leq \max_i \left\{\frac{}{}| a_i|_v \right\} \cdot \max \left\{ \frac{}{} 1, |u_1^{i_1} \cdots u_n^{i_n}|_v\right\}\\
& \leq \max_i \left\{\frac{}{}| a_i|_v \right\} \cdot \max \left\{ \frac{}{} 1, |u_1|_v^d\right\} \cdots \max \left\{ \frac{}{} 1, |u_n|_v^d\right\} \\
&= \max\{1,|u|_v\}^d \cdot  |f|_v.
\end{split}
\]
iii)  Combining part (i) and (ii) we have the following
\[
\begin{split}
|f(u \cdot x+ b)|_v & \leq 2_\infty^{2 d}\cdot \max\{1, |b|_v \}^{d} \cdot  |f_{  u}(  x)|_v \\
& \leq 2_\infty^{2d} \cdot \max\{1, |b|_v \}^{d}  \cdot  \max\{1,|u|_v\}^d  \cdot  |f|_v,  \\
\end{split}
\]
\qed

\begin{rem}  If we  convert the above bounds into   bounds for heights  we have the following. 

i) $H(f_b(x)) \leq 4^d \cdot   H(b)^d \cdot H(f)$

ii) $H(f_u(x)) \leq    H(u)^d \cdot H(f)$

iii) $H(f(ux+b)) \leq 4^d \cdot H(u)^d \cdot  H(b)^d \cdot H(f)$
\end{rem}

\proof
We  prove i) and then the rest follows in the same way.  Raising Eq. ~\eqref{eq1} to the $n_v$ power and taking the product over all valuations we have
\[ 
\begin{split} 
H_K(f_b(x)) &=\prod_{v \in M_K} |f_b(x)|_v^{n_v}\\
& \leq \prod_{v \in M_K} \left(\frac{}{} 2_\infty^{2d}  \cdot \max\{1, |  b|_v^{d}\} \cdot  |f|_v  \right)^{n_v}  \\
&\leq 2^{2d^2} \cdot H_K(b)^d  \cdot  H_K(f)
\end{split}
\]
Now, taking $[K:\Q]$-th root of both sides we obtain
\[H(f_b(x)) \leq 4^d \cdot H(b)^d \cdot H(f) .\]
\qed

%\subsection{Homogenous polynomials} 

Next we focus on homogenous polynomials. The following lemma gives a bound for the homogenous polynomial evaluated at a point.

\begin{lem}\label{hom-at-point}  Let $K$ be a number field, $f\in K[x_0, \dots, x_n]$ a homogenous polynomial of degree $d$, and $\a=(\a_0, \dots, \a_n)\in \overline K^{n+1}$. Then, the following hold:

i)    $|f(\a)|_v \leq |c(d, n)|_v  \cdot \max_j \left \{\frac{}{}|\a_j|_v\right\}^d \cdot |f|_v$, 
where $|c(d, n)|_v$ is $\binom{n+d}{d}$ is $v$ is non-Archimedean and 1 otherwise.

ii) $ H(f(\a)) \leq  c_0 \cdot  H(\a)^d \cdot H(f).$   
\end{lem}

\proof  Write $f$ as follows
\[  f(x_0, \dots, x_n)= \sum_{\substack{i_0+ \dots+ i_n =d\\ i=(i_0, \dots, i_n) } } a_ix_0^{i_0} \cdots x_n^{i_n}.\]
Let $v$ be an absolute value on $K$, extended in some way to $\overline K$. Since $f$ is a homogenous polynomial in $n$ variables of degree $d$, then the number of terms of $f$ is at most the number of monomials of degree $d$ in $n+1$ variables, and this is equal to $\binom{n+d}{n}$. We want to evaluate $H(f(\a))$.  \\
Let
\[ |c(d, n)|_v= \left \{ 
\begin{split}
& \binom{n+d}{n}\, \, \, \text{ if $v$ is Archimedean }\\
& \hspace{6mm}  1 \hspace{7mm} \text{ if $v$ is non-Archimedean} 
\end{split}
\right.
\]
then, the Gauss's norm is
\[
\begin{split}
|f(\a)|_v &=  \left|\sum_i a_i\a_0^{i_0} \cdots \a_n^{i_n}\right|_v \qquad \text{ $i=(i_0, \dots, i_n)$ and $i_0+ \cdots+ i_n =d$ }\\
%&  \leq   \sum_i \left|a_i\a_0^{i_0} \cdots \a_n^{i_n}\right|_v  \qquad \text{ by triangle inequality }\\
& \leq    |c(d, n)|_v \cdot \max_i\left\{\frac{}{}|a_i \a_0^{i_0} \cdots \a_n^{i_n}|_v \right \} \\
& \leq |c(d, n)|_v \cdot \max_j \left\{\frac{}{}|\a_j|_v \right\}^d  \cdot \max_i \left\{\frac{}{} |a_i|_v \right\}
\end{split}
\]
So we conclude,
\[|f(\a)|_v \leq |c(d, n)|_v \cdot \max_j \left\{ |\a_j|_v \right\}^d  \cdot |f|_v .\]
Taking the product over all absolute values of $K$, and then $[K:\Q]$-th root of both sides  we get the inequality  
\[H(f(\a)) \leq c_0 \cdot  H(\a)^d \cdot H(f) \]
and $c_0$ can be bounded as
\[c_0= \binom{n+d}{n}  \leq \min \left\{  (n+d)^n, 2^{n+d}  \right\}.\]
\qed

In the next session we will use Lemma~\ref{hom-at-point} to determine the height of the $SL_2 (K)$ invariants of binary forms.

\begin{cor}
Let $K$ be a number field, $f \in K[x,z]$ a homogenous polynomial of degree $d$ as follows
\[ y=f(x,z)= a_dx^d+a_{d-1}x^{d-1}z+ \dots + a_0z^d,\]
and let $\a=(\a_0, \a_1) \in \overline K^{2}$. Then,
\[H(f(\a)) \leq  \min \left \{  d+1, 2^{d+1} \right\} \cdot  H(\a)^d \cdot H(f) .\]
\end{cor}

%*****************************************************
%\newpage
\section{Heights on binary forms}\label{bin-forms}

In this section we use some of the results of the heights on polynomials to study heights on binary forms.

In this section we define the action of $ GL_2(k)$ on the space of binary forms and discuss the basic notions of their invariants. Most of this section is a summary of section 2 in \cite{vishi}. Throughout this section $k$ denotes an algebraically closed field.

Let $k[X,Z]$  be the  polynomial ring in  two variables and  let $V_d$ denote  the $(d+1)$-dimensional subspace  of  $k[X,Z]$  consisting  of homogeneous polynomials.
\begin{equation}
\label{eq1} f(X,Z) = a_0 X^d + a_1X^{d-1}Z + \dots + a_dZ^d
\end{equation}
of  degree $d$. Elements  in $V_d$  are called  \textbf{binary  forms of degree $d$}.

Since $k$ is algebraically closed, the binary form $f(X,Z)$   can be factored as
 \begin{equation}  f(X,Z)  = (z_1  X  -  x_1 Z) \cdots  (z_d  X  - x_d  Z)
=  \prod_{1 \leq  i \leq  d} det
\begin{pmatrix}
X&x_{i}\\Z&z_i\\
\end{pmatrix}
\end{equation}
The points  with homogeneous coordinates $(x_i, z_i)  \in \mathbb P^1 (k)$ are  called the \textbf{ roots  of the  binary  form} in Eq.~ \eqref{eq1}.

\subsection{Action of $GL_2(k)$ on binary forms.}
%*******************************************************

 We let $GL_2(k)$ act as a group of automorphisms on $k[X, Z]$   as follows:
\begin{equation}
 M =
\begin{pmatrix} a &b \\  c & d
\end{pmatrix}
\in GL_2(k),   \textit{   then       }
\quad  M  \begin{pmatrix} X\\ Z \end{pmatrix} =
\begin{pmatrix} aX+bZ\\ cX+dZ \end{pmatrix}.
\end{equation}
This action of $GL_2(k)$  leaves $V_d$ invariant and acts irreducibly on $V_d$. Let $A_0$, $A_1,  \dots , A_d$ be coordinate  functions on $V_d$. Then the coordinate  ring of $V_d$ can be  identified with $ k[A_0  , ... , A_d] $. For $I \in k[A_0, ... , A_d]$ and $M \in GL_2(k)$, define $I^M \in k[A_0, ... ,A_d]$ as follows
\begin{equation} \label{eq_I}
{I^M}(f):= I(M(f))
\end{equation}
for all $f \in V_d$. Then  $I^{MN} = (I^{M})^{N}$ and Eq.~(\ref{eq_I}) defines an action of $GL_2(k)$ on $k[A_0, ... ,A_d]$.

\begin{rem}
It is well  known that $SL_2(k)$ leaves a bilinear  form (unique up to scalar multiples) on $V_d$ invariant.    This form is symmetric if $d$ is even and skew symmetric if $d$ is odd.
\end{rem}

%************************************
\begin{defi}
Let  $\cR_d$  be the  ring of  $SL_2(k)$ invariants  in $k[A_0, \dots ,A_d]$, i.e., the ring of all $I \in k[A_0, \dots ,A_d]$ with $I^M = I$ for all $M \in SL_2(k)$.
\end{defi}

Note that if $I$ is an invariant, so are all its homogeneous components. So $\cR_d$ is  graded by the  usual  degree  function on  $k[A_0, \dots ,A_d]$.

Thus,  for $M  \in GL_2(k)$ we have

\begin{center}
$M(f(X,Y))  = (det(M))^ d  (z_1^{'}  X -  x_1^{'}  Z) \cdots (z_d^{'} X  - x_d^{'} Z)$.
\end{center}
where
\begin{equation}
 \begin{pmatrix}   x_i^{'}  \\  z_i^{'}  \end{pmatrix}
= M^{-1}
\begin{pmatrix} x_i\\ z_i \end{pmatrix}
\end{equation}

\begin{thm}\label{thm1} [Hilbert's Finiteness  Theorem]      $\cR_d$ is finitely  generated over $k$.   %\label{Hilbert1}
\end{thm}

   A homogeneous polynomial $I\in k[A_0, \dots , A_d, X, Y]$ is called a \textbf{covariant}  of index $s$ if
$$I^M(f)=\delta^s I(f)$$
where $\delta =\det(M)$.  The homogeneous degree in $A_1, \dots , A_n$ is called the \textbf{degree} of $I$,  and the homogeneous degree in $X, Z$ is called the \textbf{ order} of $I$.  A covariant of order zero is
called \textbf{invariant}.  An invariant is a $SL_2(k)$-invariant on $V_d$.

We will use the symbolic method of classical theory to construct covariants of binary forms. Let
\begin{equation}
\begin{split}
f(X,Z):= & \sum_{i=0}^n
\begin{pmatrix} n \\ i \end{pmatrix}a_i X^{n-i} \, Z^i, \\
 g(X,Z) := & \sum_{i=0}^m   \begin{pmatrix} m \\ i \end{pmatrix} b_i
 X^{n-i} \, Z^i    \\
\end{split}
\end{equation}
be binary forms  of  degree $n$ and $m$ respectively in $k[X, Z]$. We define the {\bf r-transvection}
\begin{equation}
(f,g)^r:= c_k \cdot \sum_{k=0}^r (-1)^k
\begin{pmatrix} r \\ k
\end{pmatrix} \cdot
\frac {\partial^r f} {\partial X^{r-k} \, \,  \partial Z^k} \cdot
\frac {\partial^r g} {\partial X^k  \, \, \partial Z^{r-k}}
\end{equation}
%\end{small}
%
where $c_k=\frac {(m-r)! \, (n-r)!} {n! \, m!}$.  It is a homogeneous  polynomial in $k[X, Z]$ and therefore a covariant of order $m+n-2r$ and degree 2. In general, the $r$-transvection of two covariants of order $m, n$ (resp., degree $p, q$) is a covariant of order $m+n-2r$  (resp., degree $p+q$).

For the rest of this paper $F(X,Z)$ denotes a binary form of order $d:=2g+2$ as below
\begin{equation}
F(X,Z) =   \sum_{i=0}^d  a_i X^i Z^{d-i} = \sum_{i=0}^d
\begin{pmatrix} n \\ i
\end{pmatrix}    b_i X^i Z^{n-i}
\end{equation}
where $b_i=\frac {(n-i)! \, \, i!} {n!} \cdot a_i$,  for $i=0, \dots , d$.  We denote invariants (resp., covariants) of binary forms by $I_s$ (resp., $J_s$) where the subscript $s$ denotes the degree (resp., the order).  
  $GL_2(k)$-invariants  are called  \textbf{absolute invariants}. They are given as ratios of $SL_2 (k)$-invariants where the numerator and denominator have the same degree.

Two binary forms $f$ and $f^\prime$ of the same degree $d$ are called \textbf{equivalent} or $GL_2 (k)$-conjugate if there is an $M \in GL_2 (k)$ such that $f^\prime = f^M$.

%Let $I_{i, j}$ be the generators of $\cR_d$ where $i$ the degree of the invariant and $j= 1, \dots , s$.  Then we have the following:

%\begin{lem}
%Two binary forms $f$ and $f^\prime$ are equivalent if and only if  there exists an $r \in k$, $r\neq 0$ such that  
%
%\[ I_{i, j} = r^i I_{i, j}^\prime, \]
%
% for all $j=1, \dots , s$. 
%\end{lem}

The main goal of this section is to determine how the height of $f^M$ changes for any given $M \in GL_2 (k)$.

\begin{lem}\label{height-bin-form}
Let $f$ be a  degree $n$ binary form 
\[f(x, z) =\sum_{i=0}^n a_i x^i z^{n-i} \]
and   $a, b, c, d \in K$ such that $ad-bc \neq 0$.  Then the following is true 
\[ \left|\frac{}{}f^M\right|_v \leq 2^{n}_v  \cdot c(n)_v \cdot \max\left\{ \frac{}{}1, |M|_v \right\}^n \cdot |f|_v .\]
\end{lem}

\proof Let us first evaluate $f(ax+bz, cx+dz) $, where $f(x, z)$ is  given and $a, b, c, d \in K^n$
\[
\begin{split}
f^M &= \sum_{i=0}^n a_i (ax+bz)^i (cx+dz)^{n-i}\\
&=\sum_{i=0}^n a_i\left( \sum_{k=0}^i \binom{i}{k}(ax)^k(bz)^{i-k}\right) \cdot \left( \sum_{l=0}^{n-i} \binom{n-i}{l} (cx)^l(dz)^{n-i-l}\right)\\
& =\sum_{\substack{k+l \leq n\\ 0 \leq k \leq n\\ 0\leq l \leq n}}\left(\sum_{k\leq i \leq n-l} a_i\binom{i}{k}\binom{n-i}{l} a^k b^{i-k} c^l d^{n-i-l}\right) \cdot x^{k+l}\cdot z^{n-(k+l)}
\end{split}
\]
Now let us estimate the Gauss's norm for this polynomial.
\[
\left|\frac{}{}f^M\right|_v = \max_{\substack{k, l\\ 0 \leq k \leq n\\ 0\leq l \leq n}}\left|\sum_{k\leq i \leq n-l} a_i\binom{i}{k}\binom{n-i}{l} a^k b^{i-k} c^l d^{n-i-l}\right|_v\\
\]
Let us denote the maximum number of terms in the sum with $c(n)$. Then  $c(n)\leq n+1 $. Estimating the binomial coefficients we have
\[\binom{i}{k}\binom{n-i}{l}\leq 2^i \cdot 2^{n-i}=2^n \]
Denote by $|M|_v= \max \left\{ \frac{}{} |a|_v, |b|_v ,|c|_v ,|d|_v\right\}$. Using these observations and notation we obtain the following estimation
\[
\begin{split}
\left|\frac{}{}f^M\right|_v  &\leq c(n)_v \cdot \max_{i,k,l}\left\{ \frac{}{} 1, \left|a_i\binom{i}{k}\binom{n-i}{l} a^k b^{i-k} c^l d^{n-i-l}\right|_v \right\}\\
&\leq c(n)_v \cdot 2^{n}_v \cdot \max_{0 \leq i \leq n}\left\{ \frac{}{}|a_i|_v\right\} \max_{0 \leq k,l \leq n}\left\{ \frac{}{}1, |a^k b^{i-k} c^l d^{n-i-l}|_v\right\}\\
& \leq 2^{n}_v \cdot c(n)_v \cdot \max_i\left\{ \frac{}{}|a_i|_v\right\} \max_{k,l}\left\{ \frac{}{}1, |a|_v^k |b|_v^{i-k} |c|_v^l |d|_v^{n-i-l}\right\}\\
&\leq 2^{n}_v \cdot c(n)_v \cdot \max_i\left\{ \frac{}{}|a_i|_v\right\} \max\left\{ \frac{}{}1, |a|_v^i |b|_v^{i} |c|_v^{n-i} |d|_v^{n-i}\right\}\\
&\leq 2^{n}_v \cdot c(n)_v  \cdot \max\left\{ \frac{}{}1, |M|_v \right\}^n \cdot |f|_v
\end{split}
\]
where $c(n)_v$ and $2_v$ are respectively $n+1$ and 2 when $v$ is Archimedean, and 1 otherwise.
\qed

\begin{thm} Let $M \in GL_2 (K)$ and  $f(x, z) \in K[x, z]$ be a degree $d$ binary form and $H(f)$ denote the absolute height of $f$.  Then,
\[ H( f^M ) \leq 2^n \cdot (n+1) \cdot H(M)^n  \cdot  H (f)\]
\end{thm}

\proof  From Lemma~\ref{height-bin-form}   for each $v \in M_k$  we have that
\[ \left|\frac{}{}f^M\right|_v \leq  2 ^{n}_v \cdot c(n)_v  \cdot \max\left\{ \frac{}{}1, |M|_v \right\}^n \cdot |f|_v.\]
Taking the product for all valuations we obtain  
\[
\begin{split}
H_K(f^M) &= \prod_{v \in M_K} \left|\frac{}{}f^M\right|_v^{n_v}\\
&\leq \prod_{v \in M_K} \left(\frac{}{} 2 ^{n}_v \cdot c(n)_v  \cdot \max\left\{ \frac{}{}1, |M|_v \right\}^n \cdot |f|_v \right)^{n_v}\\
&\leq 2^{n [K:\Q]} \cdot (n+1)^{[K:\Q]}   \cdot  H_K(M)^n  \cdot H_K (f)
\end{split}
\]
Taking $[K:\Q]$-th root we obtain the desired result. 
\qed

Next we  follow a different approach. First this technical lemma.

\begin{lem} Let $K$ be an algebraic number field, and $f \in K[x, z]$ a degree d binary form given as follows 
\[f(x, z) = \sum_{i=0}^d b_i  x^{d-i} z^i. \]
and
\[ f(u\bar x +w, \bar z) = \sum_{i=0}^d \bar{b}_i   \bar{x}^{d-i}\bar z^i.\]
 for $u, w \in K$. Then
\[\bar b_i = {d \choose i} u^{d-i}\sum_{k=0}^{i}  \frac{i! \, (d-i+k)!}{k! \, d! }b_{i-k}w^k.\]
\end{lem}

\proof We have that 
\[f(x, z) = \sum_{i=0}^d b_i  x^{d-i} z^i = \sum_{i=0}^d a_i {d \choose i} x^{d-i} z^i. \]
and 
\[
f(u\bar x +w, \bar z) = \sum_{i=0}^d \bar{b}_i   \bar{x}^{d-i}\bar z^i = \sum_{i=0}^d \bar{a}_i {d \choose i} \bar{x}^{d-i}\bar z^i.\]
where $\bar a_i = u^{d-i}\sum_{k=0}^{i} {i \choose k} a_{i-k} \, w^k$. Then,
\[
\begin{split}
\bar b_i = {d \choose i} \bar a_i &= {d \choose i} u^{d-i}\sum_{k=0}^{i} {i \choose k} a_{i-k} \, w^k\\
&={d \choose i} u^{d-i}\sum_{k=0}^{i} {i \choose k} \frac{1}{{d \choose i-k}}b_{i-k}w^k  = {d \choose i} u^{d-i}\sum_{k=0}^{i}  \frac{i! (d-i+k)!}{k! d! }b_{i-k}w^k
\end{split} \]
\qed

\begin{thm}
Let $K$ be an algebraic number field, and $f, \bar f$ as above.  The following are true:

i) For any valuation $v\in M_K$ we have 
\[ | \bar f |_v \leq 2_v^{d}  \, \cdot \, c(d)_v  \, \cdot \, |u|_v^d \, \cdot \, |w|_v^d \, \cdot \, \max_{0 \leq i \leq d}  \left\{ \frac{}{} |b_i|_v\right\} \]

ii) \[ H ( \bar f) \leq (d+1)  \, \cdot \,    2^{d} \, \cdot \, u^{d } \, \cdot \, w^{d }    \, \cdot \, H(f) \]
\end{thm}

\proof

i) For any valuation $v\in M_K$ we have the following
\[
\begin{split}
|f(ux +w, z) |_v &=  \max_{0 \leq i \leq d} \left\{ \frac{}{} |b_i|_v \right\} \\
&= \max_{0 \leq i \leq d} \left\{ \frac{}{} \left|{d \choose i} u^{d-i}\sum_{k=0}^{i}  \frac{i! (d-i+k)!}{k! d! }b_{i-k}w^k  \right|_v \right\}\\
&\leq c(d)_v \, \cdot \, \max_{0 \leq i \leq d}  \left\{ \frac{}{} \left|{d \choose i} \frac{i! (d-i+k)!}{k! d! } u^{d-i} b_{i-k} \, w^k  \right|_v \right\}\\
& =c(d)_v \, \cdot \, \max_{0 \leq i \leq d}  \left\{ \frac{}{} \left| {d+k-i \choose k} u^{d-i} \, w^k \, b_{i-k} \right|_v \right\}\\
& \leq c(d)_v \cdot \, 2^d \, \cdot  \max_{0 \leq i \leq d} \left\{ \frac{}{} 1, \left|  u^{d-i} \, w^k \, b_{i-k} \right|_v \right\}\\
&\leq c(d)_v \cdot \, 2^d \, \cdot   |u|_v^d \, \cdot |w|_v^d \, \cdot \max_i \left\{ \frac{}{} 1, |b_i|_v \right\}\\
\end{split}
\]
where $c(d)$ is the number of terms in the sum,  and $c(d)_v$ is equal to $d+1$ when $v$ is Archimedean and 1 otherwise. 

ii) Taking the product over all valuations $v\in M_K$ we have the following
\[
\begin{split}
H_K(f(ux +w, z) ) &= \prod_{v \in M_K} |f(ux +w, z)|_v^{n_v}\\
&\leq \prod_{v \in M_K} \left( 2_v^{d}  \, \cdot \, c(d)_v  \, \cdot \, |u|_v^d \, \cdot \, |w|_v^d \, \cdot \, \max_{0 \leq i \leq d}  \left\{ \frac{}{} |b_i|_v\right\} \right)^{n_v}\\
&= \left(2^{d} \, \cdot \, (d+1) \, \cdot \, u^{d} \, \cdot \, w^{d} \right)^{[K:\Q]} \prod_{v \in M_K} \max_{0 \leq i \leq d}  \left\{ \frac{}{} |b_i|_v\right\}^{n_v}\\
&= \left(2^{d} \, \cdot \, (d+1)  \, \cdot \, u^{d } \, \cdot \, w^{d } \right)^{[K:\Q]} \, \cdot \, H_K(f) 
\end{split}
\]
Taking $[K:\Q]$-th root of both sides gives the desired result. 
\qed

%*********************************************************************
\section{Minimal and moduli heights of forms}

Let $f(x,y)$ be a binary form and $Orb (f)$ its $GL_2(K)$-orbit in $V_d$. Let $H(f)$ be the height of $f$ as defined in the previous section. 

\begin{rem}There are only finitely many $f^\prime \in Orb(f)$ such that $H(f^\prime) \leq H(f)$. 
\end{rem}

Define the height of the binary form $f(x, y)$ as follows
\[ \tilde H(f) := \min \left\{ \frac{}{} H(f^\prime) | f^\prime \in Orb(f), \, H(f^\prime) \leq H(f) \right\}\]
we want to consider the following problem. For every $f$ let $f^\prime$ be the binary form such that $f^\prime \in Orb(f)$ and $\tilde H(f) = H(f^\prime)$. Determine a matrix $M \in GL_2(K)$ such that $f^\prime = f^M$.

\subsection{Moduli height of a binary form}

\def\f{\mathfrak f}  

Let $\B_d$ be the moduli space of degree $d$ binary forms defined over an algebraically closed field $k$. Then $\B_d$  is a quasi-projective variety with dimension $d-3$. We denote the equivalence class of $f$ by $\mathfrak f \in B_d$. The \textbf{moduli height} of $f(x, z)$ is defined as 
\[ \mH (f) = H (\mathfrak f)\]
where $\f $ is considered as a point in the projective space $\P^{d-3}$.  A natural question would be to investigate if the minimal height $\tilde H (f) $ has any relation to the moduli height $\mH (f)$.

Let $\left\{ I_{i, j} \right\}_{j=1}^{j=s}$ be a basis of $\cR_d$.  Here the subscript $i$ denotes the degree of the homogenous polynomial $I_{i, j}$.  The fixed field of invariants is the space $V_d^{GL_2 (K)}$ and is generated by rational functions $t_1, \dots t_r$ where each of them is a ratio of polynomials in $I_{i, j}$ such that the combined degree of the numerator is the same as that of the denominator. 

\begin{lem}
For any $SL_2 (k)$-invariant $I_i$ of degree $i$ we have that 
\[ H \left(   I_i (f) \right) \leq c \cdot H(f)^d \cdot H(I_i) \]
\end{lem}

\proof  $I_i(f)$ is a homogenous polynomial of degree $i$ evaluated at $f$.  Then the result follows from Lemma~\ref{hom-at-point}. The constant $c$ represents the number of monomials of  $I_i(f)$.  
\qed

\begin{thm}
Let $f$ be a binary form. Then,
\[ \mH (f) \leq  c \cdot \tilde H(f),  \]
for some constant $c$. 
\end{thm}

\proof
Let $\left\{ I_{i, j} \right\}_{j=1}^{j=s}$ be a basis of $\cR_d$.  Here the subscript $i$ denotes the degree of the homogenous polynomial $I_{i, j}$.  The fixed field of invariants is the space $V_d^{GL_2 (K)}$ and is generated by rational functions $t_1, \dots t_r$ where each of them is a ratio of polynomials in $I_{i, j}$ such that the combined degree of the numerator is the same as that of the denominator.

Without loss of generality we can assume that $f$ has minimal height.  So $H(f)= \tilde H (f)$.  Let $d_1, \dots , d_r$ denote the degrees of each $t_1, \dots , t_r$ respectively.  
Then,
\[ \mH (f) = H [t_1 (f), \dots , t_r (f), 1] = \prod    \, \max \{ \left| t_i (f)  \right|_v     \}_{i=1}^{i=r}.   \]
By reordering, we can assume that  
\[ \mH (f) = | t_1 (f)|_{v_1} \cdots |t_m (f)|_{v_m}  \]
However, for each $j=1, \dots m$, we have 
\[ | t_j (f) |_{v_j} \leq H(t_j) \cdot H(f), \] 
where $H(t_j) $ is a fixed constant.  This completes the proof.  
\qed

\begin{rem}
Notice that for a given degree $d$ the constant $c$ of the theorem can be explicitly computed. See for example the case of binary sextics in Section~\ref{genus2}, where this constant is 
\[ c =  2^{28} \cdot 3^9 \cdot 5^5 \cdot 7 \cdot 11 \cdot 13 \cdot  17 \cdot 43   \]
\end{rem}

%%%%%%%%%%%%%%%%%%%%%%%%%%%%%%%%%%%%%%%%%%%%%%%%%%%%%%%%%%%%%%%%%%%%%%%%
\newpage
\noindent \textbf{Part 3: Heights of algebraic curves} \\

In this lecture we focus on heights of algebraic curves.  
Our main focus is in providing equations for the algebraic curves with "small" coefficients as continuation of our previous work 
\cite{sh-1, beshaj-1, beshaj-2}.  Hence, the concept of height is the natural concept to be used.   For a genus $g\geq 2$ algebraic curve $\X_g$ defined over an algebraic number field $K$ we define the height $H_K (\X_g)$ and show that this is well-defined.  This is basically the minimum height among all curves which are isomorphic to $\X_g$ over $K$.  $\bar H_K (\X_g)$ is the height over the algebraic closure $\bar K$.  It must be noticed that our definition is on the isomorphism class of the curve and not on some equation of the curve. 
 We provide an algorithm to determine the height of a curve $C$ provided some equation for $C$.  This algorithm   is rather inefficient, but can be used for $g=2$ and $g=3$ hyperelliptic curves when the   coefficients of the initial equation of $C$ are not too large. 

\section{Heights of algebraic curves}\label{alg_curves}

In this section we want to define heights on algebraic curves given by some affine equation.  For this we will use the heights of polynomials as in Section~\ref{heights_pol}.  As before $K$ denotes an algebraic number field and $\O_K$ its ring of integers.

Let $\X_g$ be an irreducible algebraic curve with affine equation $F(x, y)=0$ for $F(x, y) \in K [x, y]$.  We define
the \textbf{height of the  curve over $K$} to be
\[H_K(\X_g):= \min \left \{    H_K(G) \, : H_K(G) \leq H_K(F) \right \}. \]
where the curve $G(x, y) =0$ is isomorphic to $\X_g$ over $K$.

If we consider the equivalence over $\bar K$ then we get another height which we denote it as $\overline H_K (\X_g)$ and call it \textbf{the height over the algebraic closure}. Namely, 
\[ \overline H_K(\X_g)= \min \{H_K(G): H_K(G) \leq H_K(F)\},\]
 where the curve $G(x, y)=0$ is isomorphic to $\X_g$ over $\overline K$.
 
In the case that $K=\Q$ we do not write the subscript $K$ and use $H(\X_g)$ or  $\overline H(\X_g)$.  Obviously, for any algebraic curve $\X_g$ we have $\overline H_K(\X_g) \leq H_K(\X_g)$.

\begin{lem} Let $K$ be a number field such that $[K:\Q] = d$.  Then, 
$H_K(\X_g)$  and   $\overline H_K(\X_g)$    are  well defined.
\end{lem}

\proof
Let $\X_g$ be an algebraic curve with affine equation $F(x, y)=0$, for $F(x, y) \in K [x, y]$. We want to show that $H_K(\X_g)$ does not depend on the choice of the polynomial $F(x, y)=0$. Let $F^\prime(x,y)=0$ be another polynomial representing our algebraic curve $\X_g$.  We can calculate $H_K(F^\prime)$ using the formula of height of a polynomial and then we search for all polynomials $G(x, y)=0$ which are isomorphic with $F^\prime(x,y)=0$ over $K$ and such that $H_K(G) \leq H_K(F^\prime)$. Then,
\[\begin{split} 
H_K(\X_g)&= \min \left \{    H_K(G) \, : H_K(G) \leq H_K(F^\prime) \right \}, \, \, \, \text{such that $G(x, y)=0$} \\
 &   \qquad   \qquad \text{is isomorphic over $K$ with $F^\prime(x,y)=0$ } \\
& =\min \left \{    H_K(G) \, : H_K(G) \leq H_K(F) \right \}, \, \, \, \text{such that $G(x, y)=0$ } \\
& \qquad \qquad \text{is isomorphic over $K$ with $F(x,y)=0$} \\
\end{split}
\]
This completes the proof. 
\qed

\begin{thm}
Let $K$ be a number field such that $[K:\Q] \leq d$. Given a constant $c$  there are only finitely many curves (up to isomorphism)  such that $H_K(\X_g) \leq c$.
\end{thm}

\proof
Let $C$ be an algebraic curve with height $H_K(C) = c$. By definition, the height of $C$ is equal to the height of a polynomial $G(x, y)=0$, i.e $H_K(G(x, y)=0)=c$. By Theorem~\ref{pol_finite} there are only finitely many polynomials with height less then $c$. Therefore, there are at most finitely many algebraic curves $\X_g$ corresponding to such polynomials with height $H_K(\X_g) \leq c$. 
\qed

%************************************
\subsection{Computing the height $H(\X_g)$ of a genus $g\geq 2$ curve $\X_g$.} 

\begin{alg} 

\textbf{Input:} algebraic curve $\X_g : F(x, y)=0$   
 $F$ has degree $d$ and is  defined over $K$

\textbf{Output:} algebraic curve $\X_g^\prime : G(x,y)=0$ such that
$\X_g^\prime \iso_K \X_g$ and $\X_g^\prime$ has minimum height.\\

\textbf{Step 1:} Compute $c_0 =H_K(F)$

\textbf{Step 2:} List all points $P \in \P^{s}(K)$ such that $H_K(P) \leq c_0$. 

\textbf{Note:}  $s$ is the number of terms of $F$ which is the number of monomials 
of degree $d$ in $n$ variables, and this is equal to $\binom{d+n-1}{d}$. From  theorem ~\eqref{thm_finite} there are only finitely many such points assume $P_1, \dots, P_r$.

\textbf{Step 3:} for $i=1$ to $r$ do 

\hspace{15mm}  Let $G_i(x, y) = p_i$;
                                 
\hspace{20mm} if $g(  G_i(x, y) ) = g(\X_g)$ then 
                                     
\hspace{24mm}  if $G_i(x, y) =0\iso_K F(x,y)=0$                                        
                                                
\hspace{28mm}   then add $G_i$ to the list $L$
                
\hspace{24mm} end if;                       

\hspace{20mm}end if;                                                                                                                                                              
                                                
\textbf{Step 4:}Return all entries of $L$ of minimum height , $L$ has curves isomorphic over $K$ to $\X_g$ of minimum height.

\end{alg}

%************************************************** 
\section{Moduli height of curves}
In this section we define the height in the moduli space of curves and investigate how this height can be used to study the curves. Our main goal is to investigate if the height of the moduli point has any relation to the height of the curve.

Let $g$ be an integer $g \geq 2$ and $\M_g$ denote the coarse moduli space of smooth, irreducible algebraic curves of genus $g$. It is known that $\M_g$ is a quasi projective variety of dimension $3g-3$.  Hence, $\M_g$ is embedded in $\P^{3g-2}$. Let $\p \in \M_g$. We call the moduli height $\mH(\p)$ the usual height $H(P)$ in the projective space $\P^{3g-2}$.  Obviously, $\mH(\p)$ is an invariant of the curve.

\begin{thm}For any constant $c\geq 1$, degree $d\geq 1$, and genus $g\geq 2$  there are finitely many superelliptic curves $\X_g$ defined over the ring of integers $\O_K$ of an algebraic number field $K$ such that   $[K:\Q] \leq d$ and  $\mH (\X_g) \leq c$.
\end{thm}

\begin{proof}
Let $\X_g$ be a genus $g$ superelliptic curve with equation  \[ y^n = x^{s+1} + a_s x^s + \dots +a_1 x + a_0,\]   defined over $K$, where $[K : \Q] \leq d$. Then,  $ H (\X_g) = H \left( P \right)$, where $P:=[a_0, \dots , a_s ] \in \P^s (K)$. From \cite[Thm. B.2.3]{silv-book} we know that there are finitely  such points in the projective space.  

To prove the result for the moduli height we consider the moduli point $\p = [\X_g]$ in the corresponding moduli space of superelliptic curves of genus $g\geq 2$. This point corresponds to a tuple $\p = [ J_0, \dots , J_r] \in \P^r (K)$ of $SL_2(K)$ invariants in the space of binary forms of degree $s$.  Again from \cite[Thm. B.2.3]{silv-book} there are only finitely many such points. 
\end{proof}

%*******************************************************************
\section{Applications to hyperelliptic and superelliptic curves}
%%%%%%%%%%%%%%%%%%%%%%%%%%%%%%%*********************************

In this section we apply some of the results above to genus 2 curves and genus 3 hyperelliptic curves. 

%*******************************************************************
%\newpage
\subsection{Genus 2 case}\label{genus2}

Let $C$ be a genus 2 curve defined over an algebraic number field $K$. Then there is a degree 2 map $\pi: C \to \P^1 (K)$, which is called the hyperelliptic projection.  Let the equation of $C$ be given by 
\[ y^2 = a_6 x^6 + \cdots + a_0\]
where $a_0, \dots , a_6 \in K$. The isomorphism classes of genus 2 curves are on one to one correspondence with the orbits of the $GL_2 (K)$-action on the space of binary sextics.  The invariant ring $\R_6$ is generated by the Igusa invariants $J_2, J_4, J_6, J_{10}$; see Section~\ref{bin-forms}  and \cite{sh-1} for details.  Note that Igusa $J$-invariants $\{J_i\}$ are homogenous polynomials of degree $i$ in   $k[a_0, \dots, a_6]$.

Let $\M_2$ be the moduli space of genus 2 curves considered as a projective variety, and $i_1, i_2, i_3$ be $GL_2(K)$-invariants given as in \cite{sh-1}. A point in $\M_2$ is given by $(i_1, i_2, i_3)$ and as a projective point by 
\[ \p =[J_4 J_2^3, (J_2 J_4 - 3 J_6) J_2^2, J_{10}, J_2^5].\] 
Notice that each $\p[i]$ is a degree 10 polynomial evaluated at $f$, i.e degree 10 polynomial  given in   $k[a_0, \cdots, a_6]$. Denote with $F_i(f) =\p[i]$. Then, from  Lemma~\ref{hom-at-point} we have
\[H    (F_i      (f) )\leq c_0 \cdot H    (F_i)  \cdot H(f)^{10}  \]
where $c_0=2^7 \cdot 3^2 \cdot 5 \cdot 7 \cdot 11 \cdot 13 \cdot  17$ is the  number of monomials of a degree 10 homogenous polynomial in seven variables.  Computations of $H(F_i)$ is done in Maple and we get
\[
H(F_1) = 2^{14} \cdot 3^7 \cdot 5^4, \, \, H(F_2)  = 2^{21} \cdot 3^{7} \cdot 5^4 \cdot 43,  H(F_3)  =  2^6 \cdot 3^5 \cdot 5, \, \, H(F_4)  = 2^{20} \cdot 3^5 \cdot 5^5  \\
\]
The maximum is $H(F_2)$. The moduli height of $f$ is computed as follows 
\[ \mH (f) = \max \{ H (F_1 (f), \dots , H( F_4 (f))\} \leq c_0 \cdot H(F_2) \cdot H(f)^{10}.  \]
Hence we have proved the following
\begin{lem}\label{bound-gen-2}
For a genus 2 curve with equation $y^2=f(x)$ the moduli height is bounded as follows
\[ \mH(f) \leq 2^{28} \cdot 3^9 \cdot 5^5 \cdot 7 \cdot 11 \cdot 13 \cdot  17 \cdot 43  \cdot H(f)^{10} \]
\end{lem}
We denote the above constant by $M_2$. From now on we write that $\mH (H) \leq M_2 H(f)^{10}$. 
Since the above result holds for any binary form equivalent to $f$ then we have that 
\[ \mH (f) \leq M_2 \cdot \tilde H (f)^{10} \]
%

%
%*********************************************************
\subsection{Genus 2 curves with height 1} 
Next we want to study genus 2 curves with height 1.  Such curves will have minimal equations with coefficients 0 or $\pm 1$. By the algorithm of the previous section we get 230 such curves listed in the Tables 1-4.  The curves are labeled 1-230 and presented by the vector of their coefficients $[ a_0, \dots , a_6]$.  From these curves 186 have automorphism groups $G$ isomorphic to the group of order 2 and are displayed in Table~\ref{tab1}.

28 of such curves have group $G$ isomorphic to $V_4$ and are displayed in Table~\ref{tab2}. There are 11 curves with automorphism group isomorphic to $D_8$ and displayed in Table~\ref{tab3}. 
The rest of the curves with larger automorphism group are displayed in Table~\ref{tab4}.

\begin{small}
\begin{table}[htdp]
\caption{Genus 2 curves with height 1 and automorphism group $D_8$}
\begin{center}
\begin{tabular}{||c|c|c|c|c|c||}
\hline 
\# &     & \# & & \# & \\
\hline 
215   &0,-1,-1,0,-1,1, 0    & 216   &  0,-1,-1,0,1,1, 0       &  217  &  0,-1,0,-1,0,-1, 0  \\
218   & 0,-1,0,-1,0,1, 0     & 219   &  -1,-1,-1,0,-1,1,-1    &  220  &  -1,-1,-1,0,1,1,1  \\
221   &-1,-1,0,0,0,-1,1    & 222   &  -1,-1,0,0,0,1,-1      &  223  &  -1, -1, 0, 0, 0, 1, 1 \\
224   & -1,-1, 1,0,1,1,-1  & 225   &  -1,0, -1,0,-1,0,-1    &       &                    \\
\hline
\end{tabular}
\end{center}
\label{tab3}
\end{table}
\end{small}

\begin{small}
\begin{table}[htdp]
\caption{Genus 2 curves with height 1 and automorphism group $ | G | \geq 10$ }
\begin{center}
\begin{tabular}{||c|c|c|c|c|c||}
\hline 
 \#   & & $|G|$  & \# & & $|G|$ \\
\hline 
226    &  1, 0, 0, 0, 0, 1, 0    & 10  & 227     &  -1, 0, 0, -1, 0, 0, -1     & 12 \\
228    &  -1, 0, 0, -1, 0, 0, 1    & 12  &  229   &  1, 0, 0, 0, 0, 0, -1    & 24  \\
230   &  0, 1, 0, 0, 0, -1, 0    & 48  &       &                          &  \\
\hline
\end{tabular}
\end{center}
\label{tab4}
\end{table}
\end{small}

\begin{tiny}
\begin{table}[htd]
\caption{Curves with height 1 and automorphism group of order 2}
\begin{center}
\begin{tabular}{||c|c|c|c|c|c|c|c||}
\hline 
\# &     & & & & & &\\
\hline 
1&-1,-1,-1,-1,-1,-1,0  & 2&-1,-1,-1,-1,-1,1,0  & 3&-1,-1,-1,-1,0,-1,0  & 4&-1,-1,-1,-1,0,1,0   \\
5&-1,-1,-1,-1,1,-1,0  & 6&-1,-1,-1,-1,1,1,0  & 7&-1,-1,-1,0,-1,-1,0  & 8&-1,-1,-1,0,-1,1,0   \\
9&-1,-1,-1,0,0,-1,0  & 10&-1,-1,-1,0,0,1,0  & 11&-1,-1,-1,0,1,-1,0  & 12&-1,-1,-1,0,1,1,0   \\
13&-1,-1,-1,1,-1,-1,0  & 14&-1,-1,-1,1,-1,1,0  & 15&-1,-1,-1,1,0,-1,0  & 16&-1,-1,-1,1,0,1,0   \\
17&-1,-1,-1,1,1,-1,0  & 18&-1,-1,0,-1,-1,-1,0  & 19&-1,-1,0,-1,-1,1,0  & 20&-1,-1,0,-1,0,-1,0   \\
21&-1,-1,0,-1,0,1,0  & 22&-1,-1,0,-1,1,-1,0  & 23&-1,-1,0,-1,1,1,0  & 24&-1,-1,0,0,-1,-1,0   \\
25&-1,-1,0,0,-1,1,0  & 26&-1,-1,0,0,0,-1,0  & 27&-1,-1,0,0,0,1,0  & 28&-1,-1,0,0,1,-1,0   \\
29&-1,-1,0,1,-1,-1,0  & 30&-1,-1,0,1,-1,1,0  & 31&-1,-1,0,1,0,-1,0  & 32&-1,-1,0,1,0,1,0   \\
33&-1,-1,0,1,1,-1,0  & 34&-1,-1,0,1,1,1,0  & 35&-1,-1,1,-1,-1,-1,0  & 36&-1,-1,1,-1,-1,1,0   \\
37&-1,-1,1,-1,0,-1,0  & 38&-1,-1,1,-1,0,1,0  & 39&-1,-1,1,-1,1,-1,0  & 40&-1,-1,1,-1,1,1,0   \\
41&-1,-1,1,0,-1,-1,0  & 42&-1,-1,1,0,-1,1,0  & 43&-1,-1,1,0,0,-1,0  & 44&-1,-1,1,0,0,1,0   \\
45&-1,-1,1,0,1,-1,0  & 46&-1,-1,1,0,1,1,0  & 47&-1,-1,1,1,-1,-1,0  & 48&-1,-1,1,1,-1,1,0   \\
49&-1,-1,1,1,0,-1,0  & 50&-1,-1,1,1,0,1,0  & 51&-1,-1,1,1,1,-1,0  & 52&-1,-1,1,1,1,1,0   \\
53&-1,0,-1,-1,-1,-1,0  & 54&-1,0,-1,-1,-1,1,0  & 55&-1,0,-1,-1,0,-1,0  & 56&-1,0,-1,-1,0,1,0   \\
57&-1,0,-1,-1,1,-1,0  & 58&-1,0,-1,-1,1,1,0  & 59&-1,0,-1,0,-1,-1,0  & 60&-1,0,-1,0,1,-1,0   \\
61&-1,0,0,-1,-1,-1,0  & 62&-1,0,0,-1,-1,1,0  & 63&-1,0,0,-1,0,-1,0  & 64&-1,0,0,-1,0,1,0   \\
65&-1,0,0,-1,1,-1,0  & 66&-1,0,0,-1,1,1,0  & 67&-1,0,1,-1,-1,-1,0  & 68&-1,0,1,-1,-1,1,0   \\
69&-1,0,1,-1,0,-1,0  & 70&-1,0,1,-1,1,-1,0  & 71&-1,0,1,-1,1,1,0  & 72&-1,0,1,0,-1,-1,0   \\
73&-1,0,1,0,1,-1,0  & 74&0,-1,-1,-1,-1,1,0  & 75&0,-1,-1,-1,0,-1,0  & 76&0,-1,-1,-1,0,1,0   \\
77&0,-1,-1,-1,1,1,0  & 78&0,-1,-1,0,0,-1,0  & 79&0,-1,-1,0,0,1,0  & 80&0,-1,-1,1,0,-1,0   \\
81&-1,-1,-1,-1,-1,-1,1  &   82&-1,-1,-1,-1,-1,0,-1  &  83&-1,-1,-1,-1,-1,0,1  &  84&-1,-1,-1,-1,-1,1,-1  \\
85&-1,-1,-1,-1,-1,1,1  &   86&-1,-1,-1,-1,0,-1,-1  &  87&-1,-1,-1,-1,0,-1,1  &  88&-1,-1,-1,-1,0,0,-1  \\
89&-1,-1,-1,-1,0,0,1  &   90&-1,-1,-1,-1,0,1,-1  &  91&-1,-1,-1,-1,0,1,1  &  92&-1,-1,-1,-1,1,-1,-1  \\
93&-1,-1,-1,-1,1,0,1  &   94&-1,-1,-1,-1,1,1,-1  &  95&-1,-1,-1,-1,1,1,1  &  96&-1,-1,-1,0,-1,-1,1  \\
97&-1,-1,-1,0,-1,0,-1  &   98&-1,-1,-1,0,-1,0,1  &  99&-1,-1,-1,0,-1,1,1  &  100&-1,-1,-1,0,0,-1,-1  \\
101&-1,-1,-1,0,0,-1,1  &   102&-1,-1,-1,0,0,0,-1  &  103&-1,-1,-1,0,0,0,1  &  104&-1,-1,-1,0,0,1,-1  \\
105&-1,-1,-1,0,0,1,1  &   106&-1,-1,-1,0,1,-1,-1  &  107&-1,-1,-1,0,1,0,-1  &  108&-1,-1,-1,0,1,0,1  \\
109&-1,-1,-1,0,1,1,-1  &   110&-1,-1,-1,1,-1,-1,1  &  111&-1,-1,-1,1,-1,0,-1  &  112&-1,-1,-1,1,-1,0,1  \\
113&-1,-1,-1,1,-1,1,1  &   114&-1,-1,-1,1,0,-1,-1  &  115&-1,-1,-1,1,0,-1,1  &  116&-1,-1,-1,1,0,0,-1  \\
117&-1,-1,-1,1,0,0,1  &   118&-1,-1,-1,1,0,1,-1  &  119&-1,-1,-1,1,0,1,1  &  120&-1,-1,-1,1,1,-1,-1  \\
121&-1,-1,-1,1,1,0,-1  &   122&-1,-1,-1,1,1,0,1  &  123&-1,-1,-1,1,1,1,-1  &  124&-1,-1,0,-1,-1,-1,1  \\
125&-1,-1,0,-1,-1,0,-1  &   \textbf{126} &-1,-1,0,-1,-1,0,1  &  127&-1,-1,0,-1,-1,1,1  &  128&-1,-1,0,-1,0,0,-1  \\
129&-1,-1,0,-1,0,0,1  &   130&-1,-1,0,-1,0,1,-1  &  131&-1,-1,0,-1,0,1,1  &  132&-1,-1,0,-1,1,-1,-1  \\
133&-1,-1,0,-1,1,0,-1  &   134&-1,-1,0,-1,1,0,1  &  135&-1,-1,0,-1,1,1,-1  &  136&-1,-1,0,0,-1,-1,1  \\
137&-1,-1,0,0,-1,0,-1  &   138&-1,-1,0,0,-1,0,1  &  139&-1,-1,0,0,-1,1,1  &  140&-1,-1,0,0,0,0,-1  \\
141&-1,-1,0,0,0,0,1  &   142&-1,-1,0,0,1,-1,-1  &  143&-1,-1,0,0,1,0,-1  &  144&-1,-1,0,0,1,0,1  \\
145&-1,-1,0,0,1,1,-1  &   146&-1,-1,0,1,-1,-1,1  &  147&-1,-1,0,1,-1,0,-1  &  148&-1,-1,0,1,-1,0,1  \\
149&-1,-1,0,1,-1,1,1  &   150&-1,-1,0,1,0,0,-1  &  151&-1,-1,0,1,0,0,1  &  152&-1,-1,0,1,1,-1,-1  \\
153&-1,-1,0,1,1,0,-1  &   154&-1,-1,0,1,1,0,1  &  155&-1,-1,0,1,1,1,-1  &  156&-1,-1,1,-1,-1,0,-1  \\
157&-1,-1,1,-1,-1,0,1  &   158&-1,-1,1,-1,-1,1,1  &  159&-1,-1,1,-1,0,0,-1  &  160&-1,-1,1,-1,0,0,1  \\
161&-1,-1,1,-1,1,0,-1  &   162&-1,-1,1,-1,1,0,1  &  163&-1,-1,1,-1,1,1,-1  &  164&-1,-1,1,0,-1,0,-1  \\
165&-1,-1,1,0,-1,0,1  &   166&-1,-1,1,0,0,0,-1  &  167&-1,-1,1,0,0,0,1  &  168&-1,-1,1,0,1,0,-1  \\
169&-1,-1,1,0,1,0,1  &   170&-1,-1,1,1,-1,0,-1  &  171&-1,-1,1,1,-1,0,1  &  172&-1,-1,1,1,0,0,-1  \\
173&-1,-1,1,1,0,0,1  &   174&-1,-1,1,1,1,0,-1  &  175&-1,-1,1,1,1,0,1  &  176&-1,0,-1,-1,-1,0,1  \\
177&-1,0,-1,-1,0,0,-1  &   178&-1,0,-1,-1,0,0,1  &  179&-1,0,-1,-1,1,0,-1  &  180&-1,0,0,-1,-1,0,1  \\
181 & 0,-1,-1,1,0,1,0 & 182 & -1,-1,0,1,0,0,-1 & 183 & -1,0,-1,0,0,-1,0  & 184 & -1,0,0,0,-1,-1,0    \\
185 & -1,0,0,0,1,-1,0    & 186 & -1,0,1,0,0,-1,0       &     &                          &  &    \\ 
\hline
\end{tabular}
\end{center}
\label{tab1}
\end{table}%
\end{tiny}

\clearpage

\begin{small}
\begin{table}[htdp]
\caption{Genus 2 curves with height 1 and automorphism group $V_4$}
\begin{center}
\begin{tabular}{||c|c|c|c|c|c||}
\hline 
\# &     & & & & \\
\hline 
187  &-1,-1,-1,-1,-1,-1,-1  & 188  &-1,-1,-1,-1,1,-1,1  & 189 &-1,-1,-1,0,-1,-1,-1  \\
190  &-1,-1,-1,0,1,-1,1     & 191  &-1,-1,-1,1,-1,-1,-1 & 192 &-1,-1,-1,1,1,-1,1  \\
193  &-1,-1,0,-1,0,-1,-1    & 194  &-1,-1,0,-1,0,-1,1   & 195& -1,-1,0,1,0,-1,-1  \\
196 &-1,-1,0,1,0,-1,1      & 197 &-1,-1,1,-1,-1,-1,1  & 198 &-1,-1,1,-1,1,-1,-1  \\
199 &-1,-1,1,0,-1,-1,1      & 200 &-1,-1,1,0,1,-1,-1   & 201 &-1,-1,1,1,-1,-1,1    \\
202 &-1,-1,1,1,1,-1,-1      & 203 &-1,0,-1,-1,-1,0,-1  & 204 &-1,0,-1,-1,1,0,1  \\
205 &-1,0,-1,0,-1,0,1       &  206 &-1,0,-1,0,0,0,-1        & 207 & -1, 0, -1, 0, 0, 0, 1 \\
208 & -1, 0, 1, -1, -1, 0,1 & 209 & -1, 0, 1, -1, 1, 0, -1 & 210 &  -1, -1, 1, 1, -1, 0 \\
211 &0,-1,-1,-1,-1,-1, 0     & \textbf{212} & 0,1,1,1,1,1,0  & 213& 0,-1,-1,0,1,-1, 0  \\
214 &0,-1,-1,1,-1,-1, 0     &     &                  &   & \\
\hline
\end{tabular}
\end{center}
\label{tab2}
\end{table}
\end{small}

Since their height is 1 then it is automatically minimal.  Compiling such tables for height $H > 1$ takes longer since there are a lot more curves and one needs to check that the height for each canditate is minimal.

From the previous Lemma~\ref{bound-gen-2} when $H (C) =1$ we get that for any curve $C$, the moduli height is $ \mH(C) \leq M_2$. 
Indeed the biggest moduli height for all 230 curves of height 1 is 
\[ \mH ( C_{126}) = 2^7 \cdot 3^2 \cdot 5^2 \cdot 151^2 \cdot 3863 < M_2= 2^{28} \cdot 3^9 \cdot 5^5 \cdot 7 \cdot 11 \cdot 13 \cdot  17 \cdot 43 \]
which occurs for curve $C_{126}$ on the table.

The curve $C$ with height $H (C) =1$ and smallest moduli height is the curve $C_{212} $ with $\mH (C) = 34560= 2^8 \cdot 3^3 \cdot 5$. The curve has equation 
\[ C_{212} : \quad y^2 = x^5+x^4+x^3+x^2+x \]
with invariants \[ i_1 = - \frac {48} 5, \quad i_2 = \frac {432} 5, \quad i_3 = 1 {400} \] and it has automorphism group isomorphic to $V_4$.

Consider now the problem of being given the moduli point as above.  Since $i_1, i_2, i_2 \in \Q$ and $\Aut (\p) \iso V_4$ then the curve is defined over $\Q$ as explained in \cite{sh-1}.  The equation provided by the algorithm in \cite{sh-1} is 
\[ 
\begin{split}
y^2 & =442765625 x^6-719030400000 x^5+320847859200000 x^4-64095440076800000 x^3\\
& +6360693303410688000 x^2-282590704159256739840 x+3449767488965367037952 \\
\end{split}
\]
and even after using Maple's "Shorten" command we only get which searches for equivalent binary forms up to transformations $x \to x+ b$ we only get 
\[ y^2 = 28337 x^6-326832 x^5+1035795 x^4-1469600 x^3+1035795 x^2-326832 x+28337 \]

In \cite{height-2} we devise an algorithm which determines the equation of superelliptic curves with minimal height.

%**************************************************************
%\newpage
\subsection{Genus 3 case}\label{genus3}

Let $C$ be a genus 3 curve defined over an algebraic number field $K$. Then there is a degree 2 map $\pi: C \to \P^1 (K)$, which is called the hyperelliptic projection.  Let the equation of $C$ be given by 
\[ y^2 = a_8 x^8 + \cdots + a_1 x + a_0\]
where $a_0, \dots , a_8 \in K$, and $\D(f) \neq 0$.  The invariant ring $R_8$ is generated by nine $SL_2(K)$-invariants $J_2, \dots, J_{10}$; see~\cite{sh-4} for details. 

Let $\M_3$ be the moduli space of genus 3 curves considered as a projective variety,  and $t_1, \dots, t_6$ be  $GL(2, k)$-invariants given  as  follows
\[
t_1:= \frac {J_3^2} {J_2^3}, \quad 
t_2:= \frac {J_4} {J_2^2}, \quad 
t_3:= \frac {J_5} {J_2\cdot J_3}, \quad 
t_4:= \frac {J_6} {J_2\cdot  J_4}, \quad 
t_5:= \frac {J_7} {J_2 \cdot J_5}, \quad 
t_6:= \frac {J_8} {J_2^4},
\]
Let 
\[\p = \left[ J_2 J_3^2 J_4 J_5, \, J_2^2 J_3 J_4^2 J_5, \, J_2^3   J_4 J_5^2, \,J_2^3 J_3  J_5 J_6, \, J_2^3 J_3 J_4 J_7,  \, J_3 J_4 J_5 J_8, J_2^4 J_3 J_4 J_5 \right]\]
be a point in $\M_3$.  Each $\p[i]$ is a degree 20 polynomial evaluated at $f$, i.e degree 20 polynomial  given in   $k[a_0, \dots, a_8]$. Denote with $F_i(f) =\p[i]$.   Then, from  Lemma~\ref{hom-at-point} we have
\[H    (F_i      (f) )\leq c_0 \cdot H    (F_i)  \cdot H(f)^{20}  \]
where $c_0= $ is the  number of monomials of a degree 20 homogenous polynomial in nine variables.  

The proof of the following lemma is provided in \cite{height-2}

\begin{lem}
For a genus 3 curve with equation $y^2=f(x)$, where $f(x)$ is a degree 8 polynomials  the moduli height is bounded as follows
\[ \mH(f) \leq c \cdot H(f)^{20} \]
\end{lem}
We denote the above constant by $M_3$. From now on we write that $\mH (H) \leq M_3 \cdot H(f)^{20}$. 
Since the above result holds for any binary form equivalent to $f$ then we have that $ \mH (f) \leq M_3 \cdot \tilde H (f)^{20}$. 
%
%%%%%%%%%%%%%%%%%%%%%%%%%%%%%%%%%%%%%%
\section{Final remarks}

A continuation of this work and proving some of the results here is intended in \cite{height-2}, where we also improve the algorithm to find an equation of the curve with minimal height.

%*********************************************************************************************************
\nocite{*}
\bibliographystyle{amsplain} 

\bibliography{mybib}{}

%*********************************************************************************************************

\end{document}